\documentclass[10pt]{article}

\usepackage{latexsym,amsfonts,amsmath,amsthm,amssymb, epsfig}
\usepackage[a4paper,total={170mm,250mm},vcentering,dvips]{geometry}
\usepackage{graphicx}%
\usepackage{mathrsfs}
\usepackage[normalem]{ulem}
\usepackage[final]{changes}
 
\usepackage{yhmath}
\usepackage{dsfont}
\usepackage{enumitem}
\usepackage{csquotes}
\usepackage{bm}
\usepackage[pdftex,dvipsnames]{xcolor}
\usepackage{float}
\usepackage{caption}
\usepackage{parskip} 
\usepackage{algpseudocode}
\usepackage{algorithm}
\usepackage{subcaption}

\usepackage[hidelinks]{hyperref}
\usepackage[noabbrev,capitalize]{cleveref}
\usepackage{tikz}
\usetikzlibrary{positioning}

\renewcommand{\autoref}{\cref}

\newtheorem{thm}{Theorem}[section]
\newtheorem{cor}[thm]{Corollary}

\newcommand{\R}{\mathbb{R}}

\newcommand{\N}{\mathbb{N}}

\newcommand{\Zdpos}{\Z^d_{+\!\!>0}}
\newcommand{\bit}{\begin{itemize}}
\newcommand{\eit}{\end{itemize}}

\newcommand{\connie}[1]{{{\color{magenta}{#1}}}}

\newcommand{\rev}[1]{{\color{blue}{#1}}}

\providecommand{\U}[1]{\protect\rule{.1in}{.1in}}
\numberwithin{equation}{section}
\newtheorem{theorem}{Theorem}[section]

\newtheorem{corollary}[theorem]{Corollary}

\newtheorem{definition}[theorem]{Definition}
\newtheorem{lemma}[theorem]{Lemma}

\newtheorem{proposition}[theorem]{Proposition}
\newtheorem*{proposition*}{Proposition}
\newtheorem{remark}[theorem]{Remark}
\newtheorem{example}[theorem]{Example}

\newtheorem{defn}[theorem]{Definition}
\newtheorem{prop}[theorem]{Proposition}

\newcommand{\bpr}{\begin{prop}}
\newcommand{\beq}{\begin{equation}}
\newcommand{\bd}{\begin{defn}}
\newcommand{\bex}{\begin{example}}
\newcommand{\bc}{\begin{cor}}
\newcommand{\bl}{\begin{lemma}}
\newcommand{\bt}{\begin{theorem}}
\newcommand{\br}{\begin{remark}}
\newcommand{\epr}{\end{prop}}
\newcommand{\eeq}{\end{equation}}
\newcommand{\ed}{\end{defn}}
\newcommand{\eex}{\end{example}}
\newcommand{\ec}{\end{cor}}
\newcommand{\el}{\end{lemma}}
\newcommand{\et}{\end{theorem}}
\newcommand{\er}{\end{remark}}

\newcommand{\C}{\mathbb C}

\newcommand{\Z}{ \mathbb{Z} }

\begin{document}

\title{Schauder and Riesz Bases for Higher-Dimensional Lebesgue Spaces Inspired by ReLU Neural Networks}
\author{
Julian Billner\footnote{Friedrich-Alexander-Universit\"at Erlangen-N\"urnberg, Mathematics Department, Cauerstr. 11, 91058 Erlangen, Germany.
Email: \href{mailto:julian.jb.billner@fau.de}{julian.jb.billner@fau.de}}
 \  and  
Cornelia Schneider\footnote{Friedrich-Alexander-Universit\"at Erlangen-N\"urnberg, Mathematics Department, Cauerstr. 11, 91058 Erlangen, Germany.
Email: \href{mailto:cornelia.schneider@fau.de}{cornelia.schneider@fau.de}}
}
\maketitle

\vskip1.5cm

\begin{abstract}
\noindent Motivated by the role of piecewise linear functions in ReLU neural networks, we study systems of dilated functions in the Lebesgue spaces $L_q([0,1]^d)$, $q\in(1,\infty)$. For a sequence $\tau\in\ell_1$ we consider the operator $T_\tau g(x)=\sum_{k=1}^\infty \tau_k g(kx)$ and the function system $\mathcal Z^\tau$ obtained by applying it to the trigonometric system. We show that $\tau\mapsto T_\tau$ is an isomorphism of algebras from $\ell_1$ with the Dirichlet convolution as multiplication onto an algebra of operators and connect the inverse of $T_\tau$ to the inverse of $\tau$ with respect to the Dirichlet convolution. We give explicit conditions on $\tau$, under which the resulting system $\mathcal Z^\tau$ is a Schauder basis of $L_q([0,1]^d)$ and, for $q=2$, a Riesz basis for which we estimate the Riesz constants in terms of $\tau$. We characterise the Riesz bases of this type in terms of a sequence space of Dirichlet convolution multipliers, and we treat mixed systems of cosine and sine functions with different sequences. Finally, we apply these results to the system generated by two piecewise linear functions inspired by ReLU networks and improve the known lower Riesz constant to $2/3$, independently of the dimension. \\[0.1cm]
\end{abstract}



\section{Introduction}\label{sect-1}

Since the middle of the twentieth century artificial neural networks have been the subject of intensive research, and in recent years they have become omnipresent in machine learning applications such as image and audio processing, where they often outperform conventional algorithms. A rigorous mathematical explanation of this success is, however, still lacking. A first step is to understand which functions can be represented or approximated efficiently by networks of a given size, and how this depends on the depth. For the ReLU activation function $x\mapsto\max\{x,0\}$ a key observation is that compositions of piecewise linear functions generate functions with a rapidly growing number of linear pieces. Iterated hat functions (sawtooth functions) were used in \cite{Telgarsky2016} to construct functions which deep networks represent with few parameters while shallow networks cannot, and in \cite{Yarotsky2017} to approximate the square function, whence products, and thereby to derive approximation rates for Sobolev functions by deep ReLU networks. We refer to \cite{DeVoreHaninPetrova2021} for a survey on neural network approximation.

Building on these ideas, Daubechies, DeVore, Foucart, Hanin and Petrova \cite{mathcalCS} studied a sequence of piecewise linear functions that can be reproduced by ReLU networks and lead to a ReLU analogue of classical Fourier analysis. This line of research was continued in \cite{2,4}, also for multivariate functions. To describe these results we recall some basic notions.

\begin{definition}[\textbf{Schauder and Riesz basis}] \label{Def RS} \
\begin{enumerate}[label=(\roman*)]
    \item A sequence $(x_n)_{n \in \N}$ of elements of a real Banach space $(X, \|\cdot\|)$ is called Schauder basis (of $X$) if for all $x \in X$ there is a unique sequence $(\alpha_n)_{n \in \N}\subset \mathbb{R}$  such that $$x= \sum \limits_{n=1}^\infty \alpha_n x_n,$$ where the convergence of the series is to be considered with respect to $\|\cdot \|.$
    \item A Schauder basis $(x_n)_{n \in \N}$ of a real Hilbert space $(H, \|\cdot\|)$ is called a Riesz basis (of $H$) if there are  Riesz constants $A\leq B \in (0, \infty)$ such that $$A \|\alpha\|_{\ell_2}^2 = A \sum \limits_{n=1}^\infty \alpha_n^2 \leq \left\|\sum \limits_{n=1}^\infty \alpha_n x_n \right\|^2 \leq B \sum \limits_{n=1}^\infty \alpha_n^2 =B \|\alpha\|_{\ell_2}^2$$ holds for all square-summable sequences $\alpha \in \ell_2.$
\end{enumerate}
\end{definition}

Next we introduce a pair of piecewise linear functions that build the base for the system introduced in \cite{DeVoreHaninPetrova2021}.

\begin{definition}[\textbf{$\mathcal C, \mathcal S$}]\label{Zacken}
    We denote by $\mathcal C$ and $\mathcal S$ the 1-periodic extensions to $\R$ of
    $$\mathcal C, \mathcal S: [0,1) \to \R, \quad \mathcal C(x):=4\left|x-\frac 1 2\right|-1, \quad \mathcal S(x):= \left|2-4\left|x-\frac 1 4\right|\right|-1.$$
\end{definition}

\begin{figure}[h]
\begin{subfigure}[c]{0.5\textwidth}
\centering
\includegraphics[width=0.7\textwidth]{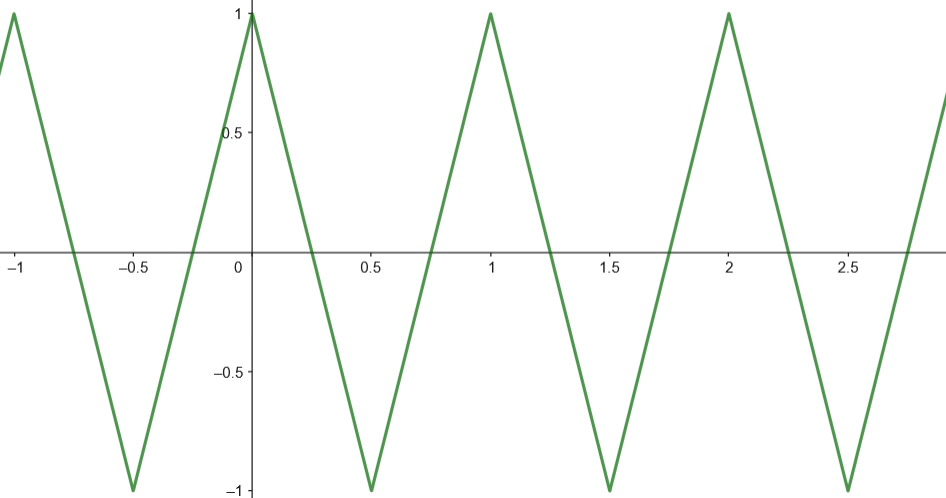}
\end{subfigure}
\begin{subfigure}[c]{0.5\textwidth}
\centering
\includegraphics[width=0.7\textwidth]{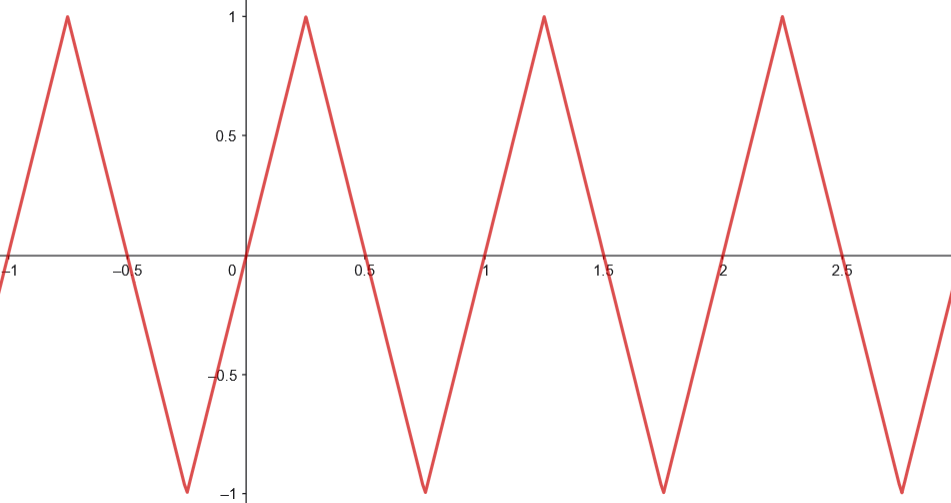}
\end{subfigure}
\caption{Plots of $\mathcal C$ (green) and $\mathcal S$ (red)}
\end{figure}

As was done with $\mathcal C$ and $\mathcal S$ in the previous definition, we extend any function $f$ defined on $[0,1)^d$ 1-periodically to $\R^d$ by setting $f(x_1, \dots, x_d):=f(x_1 - \lfloor x_1 \rfloor, \dots, x_d - \lfloor x_d \rfloor)$. Further, given a map $f:[0,1) \to \R$ (extended to $\R$ in this fashion) and a vector $\alpha \in \Z^d$ we denote by $f_\alpha$ the map \begin{equation} \label{subvector}
    f_\alpha: \R^d \to \R, \ f_\alpha(x):=f(\alpha^Tx),
\end{equation} where $\alpha^Tx = \sum_{i=1}^d \alpha_i x_i$ is the Euclidean scalar product on $\R^d.$ With these conventions we consider the function system \begin{equation} \label{Rd}
    \mathcal R_d := \{1\} \cup \{\sqrt 3 \mathcal C_J, \sqrt 3 \mathcal S_J: J \in \Zdpos\},
\end{equation}
where $\Zdpos:=\{J\in\Z^d\setminus\{0\}: J_k>0 \text{ for } k=\min\{i:J_i\neq 0\}\}$ denotes the set of integer vectors whose first non-zero entry is positive.

\textbf{Known results for $\mathcal R_d$.} The system $\mathcal R_1$ was introduced in \cite{mathcalCS}, where it was shown to be a Riesz basis of $L_2([0,1])$ with $A=\frac 12$ and $B=\frac 32$. With a different method of proof, the analogous statement for $\mathcal R_d$ in $L_2([0,1]^d)$ was established in \cite{2}. Remarkably, the multivariate system still consists of piecewise linear functions and the Riesz constants do not depend on the dimension, which avoids the curse of dimensionality. Both papers rely on the Hilbert space structure of $L_2([0,1]^d)$ and on lengthy computations of Fourier coefficients. In \cite{4} a third approach was proposed: $\mathcal R_d$ is written as the image of the trigonometric system under a bounded linear operator with bounded inverse. This approach extends to Banach spaces and shows that $\mathcal R_d$ is a Schauder basis of $L_q([0,1]^d)$ for all $q\in(1,\infty)$. Moreover, it improves the lower Riesz constant to $\frac{96}{\pi^4}\big(2-\frac{\pi^2}8\big)^2 \approx 0.5787$, again independently of $d$.

\textbf{Aim of this paper.} The approach of \cite{4} is not restricted to $\mathcal C$ and $\mathcal S$. The aim of this paper is to develop it into a general theory, which answers the question for which pairs of functions, the system of all their integer dilations is a Schauder basis of $L_q([0,1]^d)$ or a Riesz basis of $L_2([0,1]^d)$. To this end, given a sequence $\tau \in \ell_1$, we consider the linear operator \begin{equation}\label{DefTtau}
    T_\tau: L_q([0,1]^d) \to L_q([0,1]^d), \quad T_\tau g(x) := \sum \limits_{k=1}^\infty \tau_k g(kx)
\end{equation}
and the function system \begin{equation}\label{mathcalZ} \mathcal Z^\tau := \{\tau_0\} \cup \{C^\tau_J, S^\tau_J: J \in \Zdpos\} \subset L_q([0,1]^d), \end{equation} where \begin{equation} \label{Funktionentau} \tau_0=\sum \limits_{k=1}^\infty \tau_k, \quad C^\tau := \sqrt 2 \sum \limits_{k=1}^\infty \tau_k \cos( 2 \pi k\cdot), \quad \text{and} \quad S^\tau := \sqrt 2 \sum \limits_{k=1}^\infty \tau_k \sin( 2 \pi k\cdot). \end{equation}
Here, $kx=(kx_1,\dots,kx_d)$ is the usual scalar multiplication, functions on $[0,1]^d$ are extended $1$-periodically as is described above, and $C^\tau_J(x)=C^\tau(J^Tx)$, $S^\tau_J(x)=S^\tau(J^Tx)$ are defined in the sense of \eqref{subvector}. With $c(x):=\cos(2\pi x)$, $s(x):=\sin(2\pi x)$ and $e:=(1,0,0,\dots)$ we have $\tau_0=1$, $C^e_J=\sqrt 2c_J$ and $S^e_J=\sqrt 2s_J$, so $\mathcal Z^e$ is the trigonometric system, and $\mathcal Z^\tau$ is its image under $T_\tau$. The basis properties of $\mathcal Z^\tau$ are therefore determined by properties of the operator $T_\tau$. The system $\mathcal R_d$ is of this type, except that its cosine and sine parts are generated by two different sequences.
Lastly, for all $n \in \N$ we introduce the shorthand notation $G_n$ for $T_{e_n}=:G_n,$ where $(e_n)_k:=\delta_{n,k},$ that is \begin{equation} \label{Gk}
    G_n: L_q([0,1]^d) \to L_q([0,1]^d), \quad G_n g(x)=g(nx).
\end{equation} For $\tau \in \ell_1$ the operator $T_\tau$ is now equivalently given by $T_\tau = \sum \limits_{k=1}^\infty \tau_k G_k.$

For $d>1$ there is no natural way to order the system $\mathcal Z^\tau$ in order to arrive at a sequence. To deal with this issue we consider series that are indexed by $J\in \Zdpos$ and converge in the  Pringsheim sense and pick an enumeration of the underlying index set accordingly. We denote by $\ell_2(\Zdpos)$ the space of square-summable sequences with indices in $\Zdpos$ and set $\|a\|_{\ell_2}=\Big(\sum_{J \in \Zdpos} a_J^2\Big)^{1/2}.$

\textbf{Main results.} Our first result is algebraic: if $\ell_1$ is equipped with the Dirichlet convolution as multiplication, then $\tau\mapsto T_\tau$ is an isomorphism of algebras onto the abelian algebra $\{T_\tau:\tau\in\ell_1\}$ (Theorem~\ref{AlgTtau}). Consequently, if $\tau_1\neq0$ and the Dirichlet inverse $\tau^{-1}$ belongs to $\ell_1$, then $T_\tau$ is invertible with $T_\tau^{-1}=T_{\tau^{-1}}$ (Theorem~\ref{InvTtau}). This connects the question of bases to multiplicative number theory and leads to criteria on $\tau$ that are easy to verify. The following theorem summarises the main consequences. 

\begin{theorem}[Main result, informal version]\label{MainInformal}
    Let $q\in(1,\infty)$, $d\in\N$ and $\tau\in\ell_1$ such that $\|\tau\|_{\ell_1}<2|\tau_1|$, or such that $\tau_1\neq0$ and $\frac1{\tau_1}\tau$ is completely multiplicative. Let $\tau^{-1}$ denote the Dirichlet inverse of $\tau$.
    \begin{enumerate}[label=(\roman*)]
        \item $\mathcal Z^\tau$ is a Schauder basis of $L_q([0,1]^d)$. If $q=2$, it is a Riesz basis with Riesz constants $A=\|\tau^{-1}\|_{\ell_1}^{-2}$ and $B=\|\tau\|_{\ell_1}^2.$
        \item If $\lambda$ is another sequence with the same properties, with Riesz constants $A_\lambda,B_\lambda$, and $k_0\neq0$, then the mixed system $\{k_0\}\cup\{C^\tau_J,S^\lambda_J:J\in\Zdpos\}$ is a Schauder basis of $L_q([0,1]^d)$, and for $q=2$ a Riesz basis with lower constant $\min\{k_0^2,A,A_\lambda\}$ and upper constant $\max\{k_0^2,B,B_\lambda\}$.
        \item In particular, $\mathcal R_d$ is a Schauder basis of $L_q([0,1]^d)$, and a Riesz basis of $L_2([0,1]^d)$ with Riesz constants $A=\frac23$ and $B=\frac32$, independently of $d$.
    \end{enumerate}
\end{theorem}
The precise statements are Theorems~\ref{Basiscrit}, \ref{MixedZystems} and~\ref{Rdbasis}. Part~(iii) improves the lower Riesz constants $\frac12$ from \cite{mathcalCS,2} and $0.5787\dots$ from \cite{4}.

We complement these sufficient conditions by necessary ones (Lemmas~\ref{SurjTtau} and~\ref{Hilbert}), and for $q=2$ we characterise the Riesz bases of this type by means of the space $\hat\ell_d$ of Dirichlet convolution multipliers (Proposition~\ref{ellhat}).

\textbf{Related work.} Systems of dilated functions $\{\varphi(nx)\}_{n\in\N}$ have been studied extensively in connection with Dirichlet series. For $\varphi\in L_2(0,1)$, extended to $\R$ as an odd $2$-periodic function, Hedenmalm, Lindqvist and Seip \cite{Seip} showed that $\{\varphi(nx)\}_{n\in\N}$ is a Riesz basis of $L_2(0,1)$ if and only if the Dirichlet series formed by the sine coefficients of $\varphi$ and its reciprocal are multipliers of the Hilbert space $\mathcal H^2$ of Dirichlet series with square-summable coefficients. Since these multipliers are exactly the bounded analytic Dirichlet series in the half-plane $\operatorname{Re}s>0$, the condition can be expressed by the size of a function in that half-plane. Antezana, Carando and Scotti \cite{ACS21} split this condition into separate conditions for the upper and the lower Riesz bound and treated dilation systems $\{\varphi(mx,ny)\}_{m,n\in\N}$ in two variables. In \cite{Nikolski2012}, the completeness of dilation systems is related to cyclic vectors in Hardy spaces on the infinite-dimensional polydisc, and in a non-periodic version to the Riemann hypothesis. On the algebraic side, $\ell_1$ with the Dirichlet convolution is isomorphic to the Banach algebra of absolutely convergent Dirichlet series. By a theorem of Hewitt and Williamson \cite[Thm. 1]{HewittWilliamson1957}, which is the analogue of Wiener's theorem on absolutely convergent Fourier series, $\tau\in\ell_1$ has a Dirichlet inverse in $\ell_1$ if and only if $\sum_{n}\tau_n n^{-s}$ is bounded away from zero in absolute value on $\operatorname{Re}s\geq0$. Together with Theorem~\ref{InvTtau}, $T_\tau$ is therefore invertible whenever this Dirichlet series is bounded away from zero, see Corollary~\ref{CorHW}, and the conditions in Theorem~\ref{MainInformal} are easily verifiable sufficient conditions for this.

Compared with these works, our setting differs in four respects. 
First, we work with $1$-periodic generating functions and use two of them, a cosine-type function $C^\tau$ and a sine-type function $S^\tau$, together with a constant, whereas in \cite{Seip} the dilations of a single function, extended as an odd $2$-periodic function, are considered. The two settings are related by a change of variables, see Remark~\ref{CompHLS}.
Second, the multivariate systems are indexed by $\Zdpos$ and dilate along integer vectors $J$, that is $\varphi(J^Tx)$, which is different from the tensor-type dilations $\varphi(mx,ny)$. Third, our approach works in $L_q([0,1]^d)$ for all $q\in(1,\infty)$, where Hilbert space methods are not available. Fourth, it applies to mixed systems in which the cosine and the sine part are generated by different sequences, which is necessary to cover $\mathcal R_d$. For $d=1$ and $q=2$ the sine part of our characterisation reduces to the result of \cite{Seip}, see Remark~\ref{CompHLS}. 



\textbf{Organisation of the paper.} In Section~\ref{sect-2} we develop the general theory for the operators $T_\tau$ and the function systems $\mathcal Z^\tau$, including necessary conditions, mixed systems and examples as well as an connection to holomorphic functions. In Section~\ref{sect-app} we apply this theory to the system $\mathcal R_d$ and prove Theorem~\ref{MainInformal}(iii).

\section{Riesz and Schauder bases in Lebesgue spaces}\label{sect-2}

The guiding principle of this section is that basis properties are preserved under bounded linear operators with bounded inverse. Since $\mathcal Z^\tau$ is the image of the trigonometric system $\mathcal Z^e$ under $T_\tau$, it is a Schauder basis of $L_q([0,1]^d)$ whenever $T_\tau$ is bijective, and for $q=2$ a Riesz basis with Riesz constants $\|T_\tau^{-1}\|^{-2}$ and $\|T_\tau\|^2$. The main task is therefore to decide when $T_\tau$ is invertible and to estimate the norms of $T_\tau$ and $T_\tau^{-1}$. At multiple points of our research the Dirichlet convolution has made an appearance, for example, composing two operators of the type as in \eqref{DefTtau} corresponds to the Dirichlet convolution of the defining sequences, that is $T_\tau T_\lambda=T_{\tau*\lambda}$. For this reason we first recall the definition and most important properties of the Dirichlet convolution and introduce a $d$-dimensional version of it in Subsection~\ref{subsect-dirichlet}; for more details and proofs we refer to \cite[Ch.~2.6 and 2.7]{Numertheory} and \cite[Ch.~2.4]{Dirichlet}. In Subsection~\ref{subsect-T} we show that $\tau\mapsto T_\tau$ is an isomorphism of algebras and derive criteria for the invertibility of $T_\tau$. Subsection~\ref{subsect-Z} translates these into basis criteria for $\mathcal Z^\tau$. We then ask to what extent these criteria are also necessary and study the Hilbert case $q=2$ in more detail (Subsection~\ref{subsect-nec}). After that we extend our theory to mixed systems, which is needed for the investigation of $\mathcal R_d$ (Subsection~\ref{subsect-mixed}). Subsection~\ref{subsect-ex} concludes with examples and a holomorphic interpretation of the framework.


\subsection{The Dirichlet convolution}\rev{\label{subsect-dirichlet}} 

For any two real-valued sequences $a=(a_n)_{n \in \N}, b=(b_n)_{n \in \N}$ the Dirichlet convolution is usually defined as $$(a*b)_n:=\sum \limits_{\substack{c,d \in \N: \\ cd=n}} a_c b_d= \sum \limits_{d|n} a_\frac n d b_d.$$

If equipped with pointwise addition and Dirichlet multiplication, the set of complex-valued sequences forms a factorial ring, where $e$ defined by $e_n:=\delta_{1,n}$ is the neutral element and $a$ is a unit if, and only if, $a_1 \neq 0$. For our purposes it is enough to stick to real-valued sequences and to know that $\tau*a=0$ implies $a=0$ or $\tau=0$.

Let $\mu$ be the Möbius function $\mu: \N \to \R$ that is defined as follows: One sets $\mu_1:=1.$ For any $n \in \N_{\geq 2}$ determine the prime factorisation $n=\prod_{j=1}^k p_j^{a_j},$ where $p_j$ are prime numbers and $a_j$ are natural numbers for all $j=1, ..., k.$ If $a_j>1$ holds for some $j=1,...,k,$ then $\mu(n)=0.$ Else $\mu(n):=(-1)^k.$ Let $a: \N \to \R$ be fully multiplicative, that is let $a_{mn}=a_m a_n$ for all $m,n \in \N$ and $a \not\equiv 0,$ then $\mu \cdot a$ is the Dirichlet inverse of $a.$ 

When we investigate higher dimensions, the ordinary Dirichlet convolution will not suffice. Therefore we need to introduce a similar concept for sequences $(a_N)_{N \in \Zdpos}$ which will be prepared in the following lemma. We fix some more notation: For all $J \in \Zdpos$ we set $J_0:=\gcd(J_1, J_2, ..., J_d)$ and for any $k \in \N$ we put $\textbf{k}:=(k,k,...,k) \in \Zdpos.$

\begin{lemma}[\textbf{Solutions of} $kJ=N$] \label{kJ=N}
    Let $N \in \Zdpos,$ then $$\{(k,J) \in \N \times \Zdpos: kJ=N\} = \{(k,J) \in \N \times \Zdpos: k|N_0, J=k^{-1} N\}.$$
\end{lemma}

\begin{proof}
    First let $kJ=N,$ then $N_0=\gcd(kJ_1,...,kJ_d)=k \gcd(J_1,...,J_d)=kJ_0$ shows $k|N_0.$ Now $k$ divides every entry of $N,$ whence $J=k^{-1}N \in \Zdpos.$ Further, $k k^{-1}N=N$ clearly holds for all divisors $k$ of $N_0$. 
\end{proof}

Because of this lemma, the following notion of a $d$-dimensional Dirichlet convolution is well-defined.

\begin{definition}[$d$\textbf{-Dirichlet convolution}]
    Let $\tau: \N \to \R$ and $a:\Zdpos \to \R$ be sequences, then $\tau*_da: \Zdpos \to \R$ is called the $d$-Dirichlet convolution and is defined for all $N \in \Zdpos$ as $$(\tau*_da)_N:=\sum \limits_{\substack{(k,J) \in \N \times \Zdpos \\ kJ=N}} \tau_k a_J= \sum \limits_{k| N_0} \tau_k a_{k^{-1}N}.$$
\end{definition}

This definition extends the usual notion of the Dirichlet convolution to sequences on $\Zdpos$ because it holds $\Z^1_+=\N$ and $*_1=*.$ Our next goal is to investigate the relation between $*$ and $*_d, d>1$ in order to make the theory regarding $*$ available to arbitrary dimensions. Whenever we do not mention the dimension $d$ explicitly, we assume $d=1,$ denote $*_1$ as $*$ and refer to it as the (ordinary) Dirichlet convolution. Also, whenever we speak of the Dirichlet inverse of a sequence, we implicitly assume $d=1.$

\begin{lemma}[\textbf{Properties of the $d$-Dirichlet convolution}]\label{dDirichlet}
Let $\tau, \lambda: \N \to \R$ and $a: \Zdpos \to \R.$ Fix $N \in \Zdpos$ and put $\frac N{N_0}:=N_0^{-1}N \in \Zdpos$ as well as $a':\N \to \R, \ a'_k:=a_{k \frac N {N_0}}.$
\begin{enumerate}[label=(\roman*)]
    \item It holds $(\tau *_d a)'=(\tau*a')$ and $(\tau*_da)_N=(\tau*a')_{N_0}.$
    \item For all $d \in \N$ the $d$-Dirichlet convolution is pseudo-associative, that is it holds $(\tau * \lambda) *_da=\tau*_d(\lambda*_da).$
    \item It holds $e*_da=a.$
    \item It holds $\tau*_da=0,$ if and only if, it holds $a=0$ or $\tau=0.$
\end{enumerate}
\end{lemma}

\begin{proof}
\begin{enumerate}[label=(\roman*)]
    \item We have $(\tau *_d a)_k'=(\tau*_da)_{k \frac N {N_0}} = \sum \limits_{j|k} \tau_j a_{\frac k j \frac N {N_0}} = \sum \limits_{j|k}\tau_j a_{\frac k j}'=(\tau*a')_k$ for all $k \in \N.$ Further, it holds $(\tau*_da)_N = \sum \limits_{k|N_0} \tau_k a_{\frac {N_0} k \frac N {N_0}}=\sum \limits_{k|N_0} \tau_k a'_{\frac{N_0} k}=(\tau*a')_{N_0}.$ 
    \item We use the associativity of $*$. Let $N \in \Zdpos$, then it holds $$((\tau*\lambda)*_da)_N= ((\lambda*\tau)*a')_{N_0}=(\lambda*(\tau*a'))_{N_0}=(\lambda*(\tau*_da)')_{N_0}=(\lambda*_d(\tau*_da))_N.$$
    \item For any $N \in \Zdpos$ we calculate $(e*_da)_N=(e*a')_{N_0}=(a')_{N_0}=a_N$.
    \item For the if part we calculate $(0*_da)_N=(0*a')_{N_0}=0$ as well as $(\tau*_d0)_N=(\tau*0')_{N_0}=(\tau*0)_{N_0}=0$. Regarding the only if part let $\tau \neq 0.$ For any $N \in \Zdpos$ we have $0=0'=(\tau*_da)'=(\tau*a'),$ showing $a'=0$ and in particular $0=a_{N_0}'=a_N.$
\end{enumerate}
\end{proof}

\subsection{The map $T_{\tau}$ and its properties}\rev{\label{subsect-T}}
We now turn to the  class of linear operators $T_{\tau}$ from \eqref{DefTtau} and show that they have a very nice structure. 

\begin{theorem}[\textbf{Sum and scale operators $T_\tau$}] \label{Ttau}
    Let $q \in (1,\infty), \ d \in \N$ and $\tau \in \ell_1.$ Consider the operator $T_\tau$ from \eqref{DefTtau}.  Then it holds:
    \begin{enumerate}[label=(\roman*)]
        \item $T_\tau$ is a well-defined linear operator and bounded, where $\|T_\tau\|\leq\|\tau\|_{\ell_1}.$
        \item Assume additionally $\|\tau\|_{\ell_1} < 2|\tau_1|,$ then $T_\tau$ is bijective and it holds $\|T_\tau^{-1}\| \leq \frac{1}{2|\tau_1|-\|\tau\|_{\ell_1}}.$
        \item Assume $\|\tau\|_{\ell_1}<2|\tau_j|$ for some $j \in \N,$ then $T_\tau$ is bijective onto a closed subspace and $\|T_\tau^{-1}\| \leq \frac 1 {2|\tau_j|-\|\tau\|_{\ell_1}}$.
        \item Let $h,f \in L_q([0,1])$ such that $T_\tau h = f,$ then $T_\tau (h_\alpha) = f_\alpha$ holds for all $\alpha \in \Zdpos.$ 
    \end{enumerate} 
\end{theorem}

\begin{proof}
\begin{enumerate}[label=(\roman*)]
    \item Put $I:=[0,1]^d.$ For each $g \in L_q(I)$ it holds \begin{equation}\label{1} \|G_k g\|_q^q =\int_I |g(kx)|^q \ \mathrm d x = \frac 1 {k^d} \int_{[0,k]^d} |g(x)|^q \ \mathrm d x = \frac 1 {k^d} \sum \limits_{j \in \{0, ..., k-1\}^d} \int_{j+I} |g(x)|^q \ \mathrm d x =\int_I |g(x)|^q \ \mathrm d x =\|g\|_q^q \end{equation} and consequently $$\|T_\tau g\|_q \leq \sum \limits_{k=1}^\infty |\tau_k| \|g\|_q = \|\tau\|_{\ell_1} \|g\|_q.$$ In particular, $T_\tau$ is a well-defined linear operator bounded by $\|\tau\|_{\ell_1}.$
    \item Using (i), we set $\kappa=\tau - (\tau_1, 0,0, ...) \in \ell_1$ and derive $$\|T_\tau-\tau_1 \text{Id}\| = \|T_\kappa\| \leq \|\kappa\|_{\ell_1}=\|\tau\|_{\ell_1}-|\tau_1|<|\tau_1|.$$ Now $T_\tau= \tau_1 \text{Id} - (\tau_1 \text{Id}-T_\tau)$ is boundedly invertible by \cite[Ex. 5.7.1]{AppliedFuncAna} and from the proof in combination with \cite[Prop 3.4]{FuncAnaAppl} we arrive at the estimate $$ \|T_\tau^{-1}\| \leq \frac 1{|\tau_1|- \|T_\tau-\tau_1 \text{Id}\|} \leq \frac{1}{2|\tau_1|-\|\tau\|_{\ell_1}}.$$
    \item Due to \eqref{1} we calculate that $$\|T_\tau g\|_{L_q([0,1]^d)} \geq |\tau_j| \|g\|_{L_q([0,1]^d)} - \sum \limits_{k=1, k \neq j}^\infty |\tau_k| \|g\|_{L_q([0,1]^d)} = (2|\tau_j|-\|\tau\|_{\ell_1})\|g\|_{L_q([0,1]^d)}$$ holds for all $g \in L_q([0,1]^d).$ This shows injectivity and $\|T_\tau^{-1}\| \leq (2|\tau_j|-\|\tau\|_{\ell_1})^{-1},$ where $T_\tau^{-1}$ is to be understood as the inverse of $T_\tau: L_q([0,1]^d) \to \operatorname{ran}(T_\tau).$ 

    Regarding closedness of the image, let $y_n=T_\tau x_n$ be a sequence converging to $y \in L_q([0,1]^d),$ then $(y_n)_{n \in \N}$ is a Cauchy sequence. It follows $\|x_n-x_m\|=\|T_\tau^{-1}(y_n-y_m)\| \leq \|T_\tau^{-1}\|\|y_n-y_m\| \to 0$ if $n,m \to \infty.$ Whence $(x_n)_{n \in \N}$ is itself a Cauchy sequence and converges to some $x \in L_q([0,1]^d).$ Continuity of $T_\tau$ shows $y=T_\tau x \in \text{ran}(T_\tau).$ 
    \item For all $x \in I$ it holds $$T_\tau (h_\alpha)(x)=\sum \limits_{k=1}^\infty \tau_k h_\alpha(kx)=\sum \limits_{k=1}^\infty \tau_kh(\alpha^T(kx)) = \sum \limits_{k=1}^\infty \tau_k h(k(\alpha^Tx)) = T_\tau h (\alpha^Tx) = f(\alpha^Tx) = f_\alpha(x).$$
\end{enumerate}
\end{proof}

Note that in the inequality $\|\tau\|_{\ell_1} < 2|\tau_j|$ both sides depend on $|\tau_j|.$ In order to avoid this, one can use the equivalent inequality $\sum \limits_{k=1, k \neq j}^\infty |\tau_k| < |\tau_j|.$

The operators $T_\tau$ interact nicely with each other.

\begin{theorem}[\textbf{Algebraic structure of $\{T_\tau:\tau \in \ell_1\}$}] \label{AlgTtau}

Let $\mathcal A:=\{T_\tau: \tau \in \ell_1\} \subset \mathcal L(L_q([0,1]^d),$ then $\mathcal A$ is an abelian sub-algebra of $\mathcal L(L_q([0,1]^d).$ Further, the map $A: \ell_1 \to \mathcal A$, $\tau \mapsto T_{\tau}$, is an isomorphism of algebras with $\|A\|=1,$ if $\ell_1$ is equipped with the Dirichlet convolution as multiplication. 
\end{theorem}

\begin{proof}
Let $\tau, \lambda \in \ell_1.$ Clearly, for any $k \in \R$ we have $T_\tau + k T_\lambda = T_{\tau + k \lambda} \in \mathcal A,$ showing that $\mathcal A$ is a subspace and $A$ is linear. We will now prove that $\tau * \lambda \in \ell_1$ and $T_\tau \circ T_\lambda=T_{\tau*\lambda}$. It holds \begin{equation} \label{Submult}
    \|a*b\|_{\ell_1}=\sum \limits_{k=1}^\infty |(a*b)_k| \leq \sum \limits_{k=1}^\infty \sum \limits_{\substack{c,d \in \N \\ cd=k}}|a_c| |b_d| = \left(\sum \limits_{c=1}^\infty |a_c|\right) \left(\sum \limits_{d=1}^\infty |b_d|\right) = \|a\|_{\ell_1}\|b\|_{\ell_1},
\end{equation} whence $\ell_1$ is closed under multiplication. 

From \eqref{1} it follows that $G_k$ is an isometry for all $k \in \N$, whence $\|G_k\|=1.$ Obviously $G_k \circ G_j=G_{kj}$ holds and we obtain $$ T_\tau \circ T_\lambda = \left(\sum \limits_{k=1}^\infty \tau_k G_k\right) \circ \left(\sum \limits_{j=1}^\infty \lambda_j G_j\right) = \sum \limits_{k=1}^\infty \sum \limits_{j=1}^\infty \tau_k \lambda_j G_{kj}.$$ The  series on the right hand side converges absolutely because of $\sum \limits_{k=1}^\infty \sum \limits_{j=1}^\infty |\tau_k| |\lambda_j|= \sum \limits_{k=1}^\infty |\tau_k| \|\lambda\|_{\ell_1} = \|\tau\|_{\ell_1} \|\lambda\|_{\ell_1}.$ Thus, we may rearrange this expression to $\sum \limits_{n=1}^\infty \left(\sum \limits_{\substack{j,k \in \N: \\ jk=n}} \tau_k \lambda_j\right)G_n = \sum \limits_{n=1}^\infty (\tau*\lambda)_n G_n= T_{\tau * \lambda}.$ 
In particular, because $*$ is commutative it follows $T_\tau \circ T_\lambda=T_{\tau*\lambda}=T_{\lambda*\tau}=T_\lambda \circ T_\tau$ and $\mathcal A$ is an abelian sub-algebra.

We have already seen that $\|T_\tau\| \leq \|\tau_1\|_{\ell_1}$ holds for all $\tau \in \ell_1,$ showing that $\|A\| \leq 1.$ Because $\|T_e\|=\|\text{Id}\|=1$ and $\|e\|_{\ell_1}=1,$ we get $\|A\| \geq 1$ and whence $\|A\|=1.$ 

It is clear that $A$ is a surjective homomorphism of algebras. In order to prove injectivity let $v:=(1,...,1)$ be the vector of ones, then $kv=(k,...,k)$ for all $k \in \N.$ We notice that $T_\tau(\cos(2 \pi v^T\cdot))=\sum \limits_{k=1}^\infty \tau_k \cos( 2 \pi (kv)^T\cdot),$ hence $\int_{[0,1]^d} \cos(2 \pi (kv)^Tx) T_\tau(\cos(2 \pi v^T\cdot)) \mathrm d x=\frac{\tau_k} 2.$ Now $T_\tau=T_\lambda$ implies $$\tau_k=2\int_{[0,1]^d} \cos(2 \pi (kv)^Tx) T_\tau(\cos(2 \pi v^T\cdot)) \mathrm d x=2\int_{[0,1]^d} \cos(2 \pi (kv)^Tx) T_\lambda(\cos(2 \pi v^T\cdot)) \mathrm d x=\lambda_k$$ 
for all $k \in \N,$ which in turn implies $\tau=\lambda$ and proves that $A$ is injective and consequently bijective.
\end{proof}

Theorem \ref{AlgTtau} shows that the algebraic interplay of the operators $T_\tau$ is governed by that of the sequences $\tau$. For our purposes the crucial question is when $T_\tau$ is invertible, because then it maps the trigonometric system onto a Schauder basis (see Lemma~\ref{Mainbases} below). The following theorem characterises injectivity of $T_\tau$ and shows that a summable Dirichlet inverse $\tau^{-1}$ yields a bounded inverse $T_\tau^{-1}=T_{\tau^{-1}}$.

\begin{theorem}[\textbf{Invertibility of }$T_\tau$] \label{InvTtau}
    Let $q \in (1,\infty), d \in \N$ and $\tau \in \ell_1.$
    \begin{enumerate}[label=(\roman*)]
        \item The operator $T_\tau: L_q([0,1]^d) \to L_q([0,1]^d)$ is injective if, and only if, $\tau_0 \neq 0$.
        \item Let $\tau_1 \neq 0$ and $\tau^{-1}$ be its Dirichlet inverse. Assume $\tau^{-1} \in \ell_1,$ then $T_\tau$ is bijective and $T_\tau^{-1}=T_{\tau^{-1}}.$
    \end{enumerate}
\end{theorem}

\begin{proof}
\begin{enumerate}[label=(\roman*)]
    \item First let $\tau_0=0,$ then $T_\tau(1)=0=T_\tau(0)$ and $T_\tau$ is not injective. Thus, assume $\tau_0 \neq 0.$
    
    Let $h \in \ker(T_\tau),$ that is, $h \in L_q([0,1]^d)$ with $T_\tau(h)=0.$ We write $h=h_0 +\sum \limits_{J \in \Zdpos}a_J \sqrt 2 c_J + b_J \sqrt 2 s_J$ from which $$0=T_\tau h = \tau_0 h_0 + \sum \limits_{J \in \Zdpos} (\tau *_d a)_J \sqrt 2 c_J + (\tau *_d b)_J \sqrt 2 s_J$$ and $\tau *_d a=0_{\Zdpos}=\tau *_d b$ as well as $\tau_0 h_0=0$ follow. See Proposition \ref{Coeff}(i) below for a detailed proof of these formulae. 
    This implies $h_0=\frac {\tau_0 h_0}{\tau_0}=0$ as well as $a=0=b$ because $\tau \neq 0.$ We conclude $h=0,$ whence $\ker(T_\tau)=\{0\}$ and consequently shows that $T_\tau$ is injective.  
    \item It holds $T_{\tau^{-1}} \in \mathcal L(L_q([0,1]^d))$ and $T_\tau \circ T_{\tau^{-1}}=T_e=\text{Id}=T_{\tau^{-1}} \circ T_\tau.$
\end{enumerate}
\end{proof}

We will now discuss these results. A sequence $\tau$ admits a Dirichlet inverse if, and only if, $\tau_1 \neq 0$ holds. Because of Theorem \ref{AlgTtau} 
we might conjecture that for $\tau \in \ell_1$ the operator $T_\tau$ is bijective if, and only if, $\tau$ is invertible with respect to the Dirichlet convolution which is equivalent to $\tau_1 \neq 0.$ The first assertion of the foregoing theorem however states that $\tau_0 \neq 0$ is equivalent to injectivity of $T_\tau.$ Further, following its proof we see that the kernel of $T_\tau$ contains only constant functions. If we instead interpret $T_\tau$ as an operator on $L_q([0,1]^d)_0:=\{f \in L_q([0,1]^d):\int_{[0,1]^d} f(x) \mathrm d x=0\},$ we would see that $T_\tau$ is injective unless $\tau=0.$ In particular, surjectivity and bijectivity of $T_\tau$ would be equivalent. Later, we will see, that if $\tau_1 = 0,$ then $T_\tau$ is not surjective.

We investigate the second criterion a bit more. Let $\tau \in \ell_1$ be completely multiplicative, then the Dirichlet inverse of $\tau$ is given by $\mu \cdot \tau,$ which is summable, because $\mu$ is bounded by 1. If $\tau=c \cdot \lambda,$ where $\lambda$ is completely multiplicative and $c \neq 0,$ then $\tau^{-1}= \frac 1 c \cdot\lambda^{-1}= \frac 1 c \cdot \mu \cdot \tau,$ which is again summable.

If we combine the two statements, we see that for any $\tau \in \ell_1$ where $\tau_1 \neq 0=\tau_0$ the Dirichlet inverse $\tau^{-1}$ cannot be summable. If $\tau$ and $\tau^{-1}$ would both be summable, $T_\tau$ would have to be bijective and injective in particular, contradicting $\tau_0=0$. Thus, for the sequence $(1,-1,0,0,...) \in \ell_1$ we know that the Dirichlet inverse exists, but is not an element of $\ell_1.$ In particular, the assumption that both $\tau$ and $\tau^{-1}$ are summable is crucial!


We will now establish a connection between Theorem \ref{Ttau}(ii) and Theorem \ref{InvTtau}(ii).

\begin{remark}[\textbf{The Neumann series of a Dirichlet inverse}] \label{DirichletNeumann}
    Given $\tau \in \ell_1$ with $\|\tau\|_{\ell_1}<2|\tau_1|$ we used the Neumann series 
    in Theorem \ref{Ttau}(ii) to prove that $T_\tau$ is bijective. We will now show that under this condition the Dirichlet inverse of $\tau$ is summable and $\|\tau^{-1}\|_{\ell_1} \leq \frac 1 {2|\tau_1|-\|\tau\|_{\ell_1}}$. 
    
    Because $|\tau_1|>\frac{\|\tau\|_{\ell_1}}{2} \geq 0$ we know that $\tau_1 \neq 0$ and that the 
    Dirichlet inverse must exist. It holds $$\tau=\tau_1 \cdot e-(\tau_1 \cdot e-\tau) = \tau_1 \cdot (e-(e-\tau_1^{-1} \cdot \tau)).$$ Further, $\|e-\tau_1^{-1}\cdot\tau\|_{\ell_1}=|\tau_1^{-1}| \|\tau_1e-\tau\|_{\ell_1}=\frac{\|\tau\|_{\ell_1}-|\tau_1|}{|\tau_1|}<1$ holds true. For all $n \in \N$ and $\tau:\N \to \R$ we write $\tau^{*n}:=\tau * \tau * ... * \tau$ where the right hand side consists of $n$ factors and $\tau^{*0}=e.$ Because of \eqref{Submult} it holds $\|\kappa^{*n}\|_{\ell_1} \leq \|\kappa\|^n_{\ell_1}$ for all $n \in \N_0$ and $\kappa \in \ell_1.$ Since $\ell_1$ is complete, the series $\tau_1^{-1} \sum \limits_{n=0}^\infty (e-\tau_1^{-1} \cdot \tau)^{*n}$ converges in $\ell_1$ and it is bounded by $|\tau_1^{-1}| \sum \limits_{n=0}^\infty \|e-\tau_1^{-1} \cdot \tau\|^n_{\ell_1} \leq |\tau_1^{-1}|\sum \limits_{n=0}^\infty \left(\frac{\|\tau\|_{\ell_1}-|\tau_1|}{|\tau_1|}\right)^n=\frac 1 {|\tau_1|} \frac 1 {1-\frac{\|\tau\|_{\ell_1}-|\tau_1|}{|\tau_1|}}=\frac 1 {2|\tau_1|-\|\tau\|_{\ell_1}}.$ 
    One can now proceed as in \cite[Ex. 5.7.1]{AppliedFuncAna} 
    in order to see that the limit is in fact the Dirichlet inverse.


    More interestingly the equation 
    $\tau^{-1}=\tau_1^{-1} \sum \limits_{n=0}^\infty (e-\tau_1^{-1} \cdot \tau)^{*n}$ holds true in a more general sense for all $\tau:\N \to \R$ as long as $\tau_1 \neq 0.$ Letting $S_m := \tau_1^{-1} \sum \limits_{n=0}^m(e-\tau_1^{-1}\tau)^{*n}$ it holds $S_m \to \tau^{-1}$ pointwise. To see this we calculate $$\tau S_m=S_m \tau = S_m(\tau_1 \cdot (e-(e-\tau_1^{-1} \cdot \tau)))=e-(e-\tau_1^{-1} \tau)^{*(m+1)}.$$ Now we prove $(e-\tau_1^{-1}\tau)^{*m}_k=0$ for all $1 \leq k \leq m$ inductively. The case $m=1$ is clear. 
    Assuming the statement is true for $m \in \N$ we calculate $$(e-\tau_1^{-1}\tau)^{*(m+1)}_k=((e-\tau_1^{-1}\tau)^{*m}*(e-\tau_1^{-1}\tau))_k = \sum \limits_{j|k}(e-\tau_1^{-1}\tau)^{*m}_j (e-\tau_1^{-1}\tau)_{\frac k j}.$$
    First let $k \leq m,$ then because of $j \leq k \leq m$ for every divisor $j$ of $k,$ we get $(e-\tau_1^{-1}\tau)^{*m}_j=0$ and the right hand side sum vanishes. Similarly in case $k=m+1$ all summands vanish, except possibly 
    $j=k=m+1.$ But then $(e-\tau_1^{-1}\tau)_1=0.$ From this we see the pointwise convergence of $S_m$ as well as the limit being $\tau^{-1}.$
\end{remark}

\subsection{The function system $\mathcal Z^\tau$ and its basis properties}\label{subsect-Z}

We return to the original question of bases of $L_q([0,1]^d)$. Let $c, s: \R \to \R, \ c(x):= \cos (2 \pi x), \ s(x):=\sin(2 \pi x).$

The starting point is the trigonometric system $\mathcal Z^e$, which is an orthonormal basis of $L_2([0,1]^d)$ and a Schauder basis of $L_q([0,1]^d)$ for $q\in(1,\infty)$ as seen in \cite{4}. A bounded bijective operator maps a Schauder basis to a Schauder basis and an orthonormal basis to a Riesz basis, see Lemma \ref{Mainbases}. 
The system $\mathcal Z^\tau$ is by construction the image of $\mathcal Z^e$ under $T_\tau$, whence the properties of $T_\tau$ from the previous subsection can be used to decide when $\mathcal Z^\tau$ is a basis, and to estimate the Riesz constants by $\|T_\tau\|$ and $\|T_\tau^{-1}\|$. In this subsection we first make this precise (Lemma~\ref{ZystemSRBasis}) and derive criteria in terms of $\tau$ (Theorem~\ref{Basiscrit}). We then show how the coefficients with respect to these systems can be computed with the $d$-Dirichlet convolution (Proposition~\ref{Coeff}). Necessary conditions, the Hilbert space case and mixed systems are treated in the two following subsections.

\begin{lemma}[\textbf{The main bases}] \label{Mainbases}
    Let $q \in (1,\infty), d \in \N$ and $T \in \mathcal L(L_q([0,1]^d))$ be bijective, then $$\{T(1)\} \cup \{T(\sqrt 2 c_J), T(\sqrt 2 s_J): J \in \Zdpos\} \subset L_q([0,1]^d)$$ is a Schauder basis and 
    in case $q=2$ a Riesz basis with optimal Riesz constants $A=\|T^{-1}\|^{-2}$ and $B=\|T\|^2$. 
\end{lemma}

\begin{proof}
    In 
    \cite[Thm. 2.6, Prop. 3.7]{4} it was observed that the trigonometric system \begin{equation} \mathcal \{1\} \cup \{\sqrt 2 \cos(2 \pi J^T\cdot), \sqrt 2 \sin(2 \pi J^T \cdot): J \in  \Zdpos \} \end{equation} \label{ONB} forms an orthonormal basis of $L_2([0,1]^d)$ and a Schauder basis of $L_q([0,1]^d),$ whence the statement follows from \cite[p.75]{SetsofSpecialFunctions} and \cite[Prop. 3.6.8]{Basen}.
\end{proof}

The function systems $\mathcal Z^\tau$ are exactly of this form, if we choose $T=T_\tau$ for an appropriate $\tau \in \ell_1.$ Setting $\tau=e$ we regain the orthonormal basis from \eqref{ONB}. Another instance of such a function system is the system $\mathcal R_d$ from \eqref{Rd} which will be studied more closely in Section~\ref{sect-app}. 

\begin{lemma}[\textbf{Properties of $\mathcal Z^\tau$}] \label{ZystemSRBasis}
    Let $q \in (1,\infty), d \in \N$ and $\tau \in \ell_1.$ Consider $T_\tau:L_q([0,1]^d) \to L_q([0,1]^d).$
    \begin{enumerate}[label=(\roman*)]
        \item Let $C^\tau,S^\tau$ be defined as in \eqref{Funktionentau}, then it holds $C^\tau, S^\tau \in L_q([0,1]) \cap C([0,1])$.
        \item If $T_\tau$ is bijective, $\mathcal Z^\tau$ is a 
        Schauder basis of $L_q([0,1]^d)$ and a Riesz basis of $L_2([0,1]^d)$ with Riesz constants $A=\|T_\tau^{-1}\|^{-2}$ and $B=\|T_\tau\|^2.$ 
    \end{enumerate}
\end{lemma}
\begin{proof}
\begin{enumerate}[label=(\roman*)]
    \item For all $k \in \N$ the maps $\cos(2 \pi k \cdot), \sin(2 \pi k \cdot)$ are continuous and bounded by 1. Thus, for all $x \in [0,1]$ $$|C^\tau(x)|,|S^\tau(x)| \leq \sqrt 2 \sum \limits_{k=1}^\infty |\tau_k| = \sqrt 2 \|\tau\|_{\ell_1} < \infty$$ holds and the M-test of Weierstraß ensures uniform convergence and continuity of $C^\tau, S^\tau$. In particular, this implies Lebesgue measurability and $C^\tau, S^\tau \in L_q([0,1])$ as both functions are bounded (by $\sqrt 2 \|\tau\|_{\ell_1}).$
    \item This follows directly from combining Theorem \ref{Ttau}(i) and (iv) with Lemma \ref{Mainbases} and the calculation \begin{equation}
        T_\tau(1)=\sum \limits_{k=1}^\infty \tau_k=\tau_0, \quad T_\tau(\sqrt 2 c)=\sqrt 2 \sum \limits_{k=1}^\infty \tau _k c_k=C^\tau, \quad T_\tau(\sqrt 2 s)=\sqrt 2 \sum \limits_{k=1}^\infty \tau _k s_k=S^\tau.
    \end{equation} 
\end{enumerate}
\end{proof}

We are now able to give some handy criteria for when $\mathcal Z^\tau$ results in a basis of the Lebesgue spaces $L_q([0,1]^d)$. In case $q=2$ we achieve a Riesz basis and are able to estimate the Riesz constants in terms of $|\tau_1|$ and $\|\tau\|_{\ell_1}$. Note that we always may increase the upper Riesz constant and decrease the lower one.

\begin{theorem}[Basis criteria] \label{Basiscrit}
    Let $q \in (1,\infty), d \in \N$ and $\tau \in \ell_1.$
    \begin{enumerate}[label=(\roman*)]
        \item Let $\|\tau\|_{\ell_1}<2|\tau_1|,$ then $\mathcal Z^\tau$ is a Schauder basis of $L_q([0,1]^d).$ If $q=2$ it is a Riesz basis with Riesz constants $A=\frac 1 {\|\tau^{-1}\|_{\ell_1}^2} \geq (2|\tau_1|-\|\tau\|_{\ell_1})^2$ and $B=\|\tau\|_{\ell_1}^2.$
        \item Let $\tau_1 \neq 0$ and $\frac 1 {\tau_1} \cdot \tau$ be completely multiplicative. Then $\mathcal Z^\tau$ is a Schauder basis of $L_q([0,1]^d).$ If $q=2$ it is a Riesz basis with Riesz constants $A=\frac 1 {\|\tau^{-1}\|_{\ell_1}^2}
        =\frac{|\tau_1|^4}{\|\mu \cdot \tau\|_{\ell_1}^2} \geq\frac{|\tau_1|^4}{\|\tau\|_{\ell_1}^2}$ and $B=\|\tau\|_{\ell_1}^2,$ where $\mu$ is the Möbius function. 
    \end{enumerate}
\end{theorem}

\begin{proof}
    \begin{enumerate}[label=(\roman*)]
        \item This follows from the combination of Lemma \ref{ZystemSRBasis}(ii), Remark \ref{DirichletNeumann} and Theorem \ref{Ttau}(ii) and (iii).
        \item By \cite[Thm. 2.17]{Numertheory} the Dirichlet inverse of $\frac 1 {\tau_1} \cdot \tau$ is given by $\mu \cdot \frac 1 {\tau_1} \cdot \tau,$ therefore we may calculate the inverse of $\tau$ via $\tau^{-1}=(\tau_1 \cdot \frac 1 {\tau_1} \cdot \tau)^{-1}=\frac 1 {\tau_1^2} \cdot \mu \cdot \tau,$ for which it holds $$\|\tau^{-1}\|_{\ell_1}= \sum \limits_{k=1}^\infty \frac 1 {|\tau_1|^2} \cdot |\mu_k| \cdot |\tau_k| = \frac 1 {|\tau_1^2|} \|\mu \cdot \tau\|_{\ell_1} \leq \frac 1 {|\tau_1|^2} \|\tau\|_{\ell_1}< \infty$$ because $\mu$ is bounded by 1. Whence, $\tau$ and $\tau^{-1}$ are summable. By Theorem \ref{InvTtau}(ii) the operator $T_\tau$ is bijective and $T_\tau^{-1}=T_{\tau^{-1}}$ gives $\|T_\tau^{-1}\|\leq \|\tau^{-1}\|_{\ell_1} = \frac{\|\mu \cdot \tau\|_{\ell_1}}{|\tau_1|^2} \leq \frac{\|\tau\|_{\ell_1}}{|\tau_1|^2}.$ The rest follows from Lemma \ref{ZystemSRBasis}(ii).
    \end{enumerate}
\end{proof}

The two criteria of Theorem \ref{Basiscrit} are special cases of a characterisation of the sequences $\tau\in\ell_1$ that are invertible with respect to the Dirichlet convolution, which is due to Hewitt and Williamson. For $\tau\in\ell_1$ we consider the Dirichlet series $F_\tau(s):=\sum_{n=1}^\infty\tau_nn^{-s}$ for $s\in\C$ with $\operatorname{Re}s\geq0$ for which it converges absolutely.

\begin{corollary}[\textbf{A Wiener-type basis criterion}]\label{CorHW}
Let $q\in(1,\infty)$, $d\in\N$, and $\tau\in\ell_1$. The following two assertions are equivalent.
\begin{enumerate}[label=(\roman*)]
    \item $\inf_{\operatorname{Re}s\geq0}|F_\tau(s)|>0$.
    \item $\tau_1\neq0$ and the Dirichlet inverse $\tau^{-1}$ belongs to $\ell_1$.
\end{enumerate}
In case either assertion holds true the operator $T_\tau:L_q([0,1]^d) \to L_q([0,1]^d)$ is bijective with $T_\tau^{-1}=T_{\tau^{-1}}$, $\mathcal Z^\tau$ is a Schauder basis of $L_q([0,1]^d)$, and for $q=2$ it is a Riesz basis with Riesz constants $A=\|\tau^{-1}\|_{\ell_1}^{-2}$ and $B=\|\tau\|_{\ell_1}^2$.
\end{corollary}

\begin{proof}
The equivalence of (i) and (ii) is \cite[Thm. 1]{HewittWilliamson1957}. Note that the inverse in $\ell_1$ is the Dirichlet inverse, which is real if $\tau$ is real, and that $\tau_1\neq0$ is equivalent to its existence. If (ii) holds, Theorem \ref{InvTtau}(ii) shows that $T_\tau$ is bijective with $T_\tau^{-1}=T_{\tau^{-1}}$. By Lemma \ref{ZystemSRBasis}(ii), $\mathcal Z^\tau$ is a Schauder basis of $L_q([0,1]^d)$ and for $q=2$ a Riesz basis with constants $\|T_\tau^{-1}\|^{-2}$ and $\|T_\tau\|^2$. Since $\|T_\tau\|\leq\|\tau\|_{\ell_1}$ and $\|T_\tau^{-1}\|=\|T_{\tau^{-1}}\|\leq\|\tau^{-1}\|_{\ell_1}$ by Theorem \ref{Ttau}(i), the constants $A=\|\tau^{-1}\|_{\ell_1}^{-2}$ and $B=\|\tau\|_{\ell_1}^2$ are admissible.
\end{proof}

The hypothesis of Theorem \ref{Basiscrit}(i) directly implies condition (i) of Corollary \ref{CorHW}: If $\|\tau\|_{\ell_1}<2|\tau_1|$, then $|F_\tau(s)|\geq|\tau_1|-\sum_{n\geq2}|\tau_n|=2|\tau_1|-\|\tau\|_{\ell_1}>0$ for $\operatorname{Re}s\geq0$. Condition (i) is often easy to verify when $F_\tau$ has a product representation, for instance an Euler product, which is the case if $\frac1{\tau_1}\tau$ is completely multiplicative.

\begin{remark}{\textbf{(Repeated transforms of $\mathcal Z^\tau$).}}
    In order to obtain Riesz and Schauder bases of the Lebesgue spaces we chose the orthonormal basis $\mathcal Z^e= \{1\} \cup \{\sqrt 2 \cos(2 \pi J^T\cdot), \sqrt 2 \sin(2 \pi J^T \cdot): J \in 
    \Zdpos\}$ as a starting point and transformed this system using a bijective operator $T_\tau$ resulting in the system $\mathcal Z^\tau.$ Since Riesz and Schauder bases also are suitable starting points for a transformation, one might consider to transform $\mathcal Z^\tau$ again with another bijective operator $T_\lambda$ in order to achieve new systems unobtainable from $\mathcal Z^e$ in a single transformation. However, due to Theorem \ref{AlgTtau} this 
    is not necessary as the transformation of $\mathcal Z^\tau$ with $T_\lambda$ results in the transformation of $\mathcal Z^e$ with the map $T_{\tau*\lambda}.$ However, additional restrictions like $2|\tau_1|>\|\tau\|_{\ell_1}$ may be easier met by an application of multiple operators rather than a single composite operator.
\end{remark}

We next investigate the coefficients of expressions 
within the system $\mathcal Z^\tau.$

\begin{proposition}[\textbf{Coefficients}] \label{Coeff}
    Let $q \in (1,\infty), d \in \N, \tau \in \ell_1$ and $a,b:\Zdpos \to \R.$
    \begin{enumerate}[label=(\roman*)]
        \item If $\sum \limits_{J \in \Zdpos} a_J \sqrt 2 c_J + b_J \sqrt 2 s_J$ converges in $L_q([0,1]^d),$ it holds \begin{equation}
            T_\tau \Bigg(\sum \limits_{J \in \Zdpos} a_J \sqrt 2 c_J + b_J \sqrt 2 s_J \Bigg) = \sum \limits_{J \in \Zdpos} (\tau *_da)_J \sqrt 2 c_J + (\tau*_db)_J \sqrt 2 s_J.
        \end{equation}
        \item Let $\tau_1 \neq 0.$ If $f=\sum \limits_{J \in \Zdpos} a_J C^\tau_J + b_J S^\tau_J$ converges in $L_q([0,1]^d),$ it holds $a=\tau^{-1} *_d \hat a$ and $b=\tau^{-1}*_d \hat b$ where $\hat a_J:=\int_{[0,1]^d} f(x) \sqrt 2 c_J(x) \mathrm d x$ and $\hat b_J:=\int_{[0,1]^d} f(x) \sqrt 2 s_J(x) \mathrm d x$ for all $J \in \Zdpos.$
    \end{enumerate}
\end{proposition}

\begin{proof}
    \begin{enumerate}[label=(\roman*)]
        \item Because $T_\tau$ is linear and bounded, by Theorem \ref{Ttau}(i) we get \begin{align*} T_\tau\Bigg(\sum \limits_{J \in \Zdpos} a_J \sqrt 2 c_J + b_J \sqrt 2 s_J\Bigg) &= \sum \limits_{J \in \Zdpos} a_J \sqrt 2 T_\tau(c_J) + b_J \sqrt 2 T_\tau(s_J) \\ &= \sum \limits_{k=1}^\infty \sum \limits_{J \in \Zdpos} \tau_k a_J \sqrt 2 c_{kJ} + \tau_k b_J \sqrt 2 s_{kJ}\\&=\sum \limits_{N \in \Zdpos} \left(\sum \limits_{\substack{(k,J) \in \N \times \Zdpos \\ kJ=N}} \tau_k a_J \right) \sqrt 2 c_N + \left(\sum \limits_{\substack{(k,J) \in \N \times \Zdpos \\ kJ=N}} \tau_k b_J \right) \sqrt 2 s_N \\ &= \sum \limits_{N \in \Zdpos} (\tau*_da)_N \sqrt 2 c_N + (\tau*_db)_J \sqrt 2 s_N. \end{align*}
        \item The sequences $\hat a, \hat b:\Zdpos\rightarrow \mathbb{R}$ are well-defined because every integral converges by Hölder's inequality and Lemma \ref{ZystemSRBasis}(i) applied to the conjugate exponent $q'.$ The Dirichlet inverse of $\tau$ exists because $\tau_1 \neq 0$ is assumed. We calculate \begin{align*} \hat a_J=\int_{[0,1]^d} f(x) \sqrt 2 c_J(x) \mathrm d x &= \sum \limits_{N \in \Zdpos} a_N \int_{[0,1]^d} C^\tau_N(x) \sqrt 2c_J(x) \mathrm d x \\ &= \sum \limits_{N \in \Zdpos} \sum \limits_{k=1}^\infty a_N \tau_k \int_{[0,1]^d} \sqrt 2 c_{kN}(x) \sqrt 2 c_J(x) \mathrm d x = \sum \limits_{\substack{(k,N) \in \N \times \Zdpos \\ kN=J}} \tau_k a_N = (\tau *_da)_J. \end{align*} Using Lemma \ref{dDirichlet}(ii),(iii) we deduce $$a=e*_da=(\tau^{-1}*\tau)*_da=\tau^{-1}*_d(\tau*_da)=\tau^{-1}*_d\hat a.$$ Switching $a$ with $b, c$ with $s$ as well as $C^\tau$ with $S^\tau$ shows $\hat b_J=(\tau *_d b)_J$ in the same way.
    \end{enumerate}
\end{proof}

We shortly highlight another way to calculate $a$ and $b$ than first computing $\tau^{-1}$ and afterwards evaluating $\tau^{-1}*_d\hat a$ and $\tau^{-1}*_d\hat b$. The formula $\hat a = \tau *_d a$ can be written as $\int_{[0,1]^d} f(x) \sqrt 2 c_J(x) \mathrm d x=\sum \limits_{k|J_0} \tau_k a_{k^{-1}J}$ from which we subtract $\sum \limits_{k|J_0, k \neq 1} \tau_k a_{k^{-1}J}$ and afterwards divide by $\tau_1 \neq 0$ to get $$a_J=\frac{\int_{[0,1]^d} f(x) \sqrt 2 c_J(x) \mathrm d x-\sum \limits_{k|J_0, k \neq 1} \tau_k a_{k^{-1}J}}{\tau_1}.$$ 
Similarly, one sees that $$b_J=\frac{\int_{[0,1]^d} f(x) \sqrt 2 s_J(x) \mathrm d x-\sum \limits_{k|J_0, k \neq 1} \tau_k b_{k^{-1}J}}{\tau_1}.$$ Both formulas can be evaluated recursively starting with all $J \in \Zdpos$ where $J_0=1$ and afterwards evaluating those $J \in \Zdpos$ where $J_0$ is prime, a product of two, three, four, ... primes.

\subsection{Necessary conditions and the sequence spaces $\hat\ell_d$}\rev{\label{subsect-nec}}
So far we have given sufficient conditions on $\tau$ for $\mathcal Z^\tau$ to be a basis. In this subsection we turn to necessary conditions. 
We 
first turn to the question of surjectivity of the operator $T_\tau.$

\begin{lemma}[\textbf{Surjectivity of $T_\tau$}] \label{SurjTtau}
    Let $q \in (1,\infty), d \in \N$ and $\tau \in \ell_1.$ Let $T_\tau:L_q([0,1]^d) \to L_q([0,1]^d)$ be surjective, then $\mathcal Z^\tau$ is a Schauder basis of $L_q([0,1]^d)$ and it holds $\tau_1 \neq 0 \neq \tau_0.$ 
\end{lemma}

\begin{proof}
    Assume $\tau_0=0.$ Any $f \in L_q([0,1]^d)$ can be written as $f=f_0+\sum \limits_{J \in \Zdpos} a_J \sqrt 2 c_J + b_J \sqrt 2 s_J$ for some $a,b:\Zdpos \to \R.$ According to Proposition \ref{Coeff}(i) it holds $T_\tau f = \sum \limits_{J \in \Zdpos} (\tau*_da)_J \sqrt 2 c_J + (\tau*_db)_J \sqrt 2 s_J,$ whence 1 is not in the image of $T_\tau,$ a contradiction to $T_\tau$ being surjective. 
    If $\tau_0\neq 0$, it follows from Theorem \ref{InvTtau}(i)  that $T_\tau$ is injective and consequently bijective. From Lemma \ref{ZystemSRBasis}(ii) we infer that $\mathcal Z^\tau$ is a Schauder basis. Now assume $\tau_1=0$ and consider $g=\sqrt 2c_\textbf{1}.$ Because $\mathcal Z^\tau$ is a Schauder basis we may write $g=g_0 \tau_0+\sum \limits_{J \in \Zdpos} a_J C^\tau_J + b_J S^\tau_J,$ where $g_0 \in \R$ and $a,b:\Zdpos \to \R.$ Using Proposition \ref{Coeff}(ii) we calculate $b=0$ and from $0=\int_{[0,1]^d}g(x) \mathrm d x = g_0 \tau_0$ we see $g_0=0.$ Using that $g$ is bounded we arrive at the contradiction \begin{align*}
        1=\int_{[0,1]^d} g(x) g(x) \mathrm d x = \sum \limits_{J \in \Zdpos} a_J \int_{[0,1]^d} g(x) C_J^\tau(x) \mathrm d x = \sum \limits_{J \in \Zdpos} \sum \limits_{k=2}^\infty a_J \tau_k \int_{[0,1]^d} \sqrt 2 c_\textbf{1}(x) \sqrt 2 c_{kJ}(x) \mathrm d x = 0,
    \end{align*} because $kJ \neq \textbf{1}$ holds for all $k \in \N_{\geq 2}$ and $J \in \Zdpos$ according to Lemma \ref{kJ=N}. Whence $\tau_1 \neq 0.$
\end{proof}

When investigating $L_q([0,1]^d)_0$ instead we need to modify the argumentation a little bit. One would get from Theorem \ref{InvTtau}(i) that $T_\tau$ is bijective and a similar 
reasoning as in Lemma \ref{ZystemSRBasis}(ii) 
gives that  $\mathcal Z^\tau \backslash \{\tau_0\}$ is a Schauder basis of $L_q([0,1]^d)_0.$ The remaining part of the proof would then show $\tau_1 \neq 0.$ However, in this case $\tau_0$ might be zero.

We give another necessary condition in the Hilbert space case. We use the fact that if $\mathcal Z^\tau$ is a Schauder basis, then $\text{span}(\mathcal Z^d) ^\perp=L_2([0,1]^d)^\perp=\{0\}.$

\begin{lemma}[\textbf{The Hilbert case}] \label{Hilbert}
    Let $d \in \N$ and $\tau \in \ell_1.$ If $T_\tau:L_2([0,1]^d) \to L_2([0,1]^d)$ is surjective, then $$\bigg\{a \in \ell_2(\Zdpos):\sum \limits_{i=1}^\infty a_{iJ} \tau_i = 0 \text{ for all } J \in \Zdpos \bigg\}=\{0\}.$$
\end{lemma}

\begin{proof}
     Let $a \in \ell_2(\Zdpos)$ such that $\sum \limits_{i=1}^\infty a_{iJ} \tau_i = 0 \text{ for all } J \in \Zdpos$. Then $f:= \sum \limits_{K \in \Zdpos}^\infty a_K \sqrt 2 c_K$ converges in $L_2([0,1]^d).$ Further, we  calculate $$\int_{[0,1]^d} f(x) C^\tau_J(x) \mathrm d x = \sum \limits_{K \in \Zdpos}^\infty \sum \limits_{i=1}^\infty a_K \tau_i \int_{[0,1]^d} \sqrt 2 c_K(x) \sqrt 2 c_{iJ}(x) \mathrm d x= \sum \limits_{i=1}^\infty a_{iJ} \tau_i=0$$ as well as $\int_{[0,1]^d} f(x) \mathrm d x=0=\int_{[0,1]^d} f(x) S^\tau_J(x) \mathrm d x.$ Thus, $f=0$ and $a=0$ follow. 
\end{proof}

\begin{remark}[\textbf{The case} $q \neq 2$]
     The case $q \neq 2$ can be studied in a similar fashion. 
     Let $T_\tau:L_q([0,1]^d)\rightarrow L_q([0,1]^d)$  be surjective.
     Using the Hahn-Banach theorem and the characterisation of the dual space of $L_q([0,1]^d)$, one sees that the only $f \in L_{q'}([0,1]^d)$ such that $\int_{[0,1]^d} f(x) \mathrm d x=\int_{[0,1]^d} f(x) S^\tau_J(x) \mathrm d x=\int_{[0,1]^d} f(x) C^\tau_J(x) \mathrm d x=0$ for all $J \in \Zdpos$ is the zero function. Here, $q'=
     \frac{q}{q-1} \in (1,\infty)$ is the dual exponent. Since $\mathcal Z^e$ is a Schauder basis of $L_{q'}([0,1]^d)$ we can do an analogous calculation as before, however, one does not know which sequences to check because this depends on the convergence of the corresponding series in $L_{q'}([0,1]^d).$
 \end{remark}

Proposition \ref{Coeff} also allows a different approach to our setting in the Hilbert space case. 

\begin{proposition}[\textbf{The sequence spaces} $\hat\ell_d$] \rev{\label{ellhat}}
Let $$\hat\ell_d :=\left\{\tau: \N \to \R \vert \ 
\exists c=c(\tau) \in [0,\infty) \text{ such that } \|\tau*_da\|_{\ell_2} \leq c \|a\|_{\ell_2} \ 
\forall\, a \in \ell_2(\Zdpos)\right\}$$ and let $\|\tau\|_{\hat\ell_d}$ be the infimum of all such $c.$ 

\begin{enumerate}[label=(\roman*)]
\item $(\hat\ell_d,\|\cdot\|_{\hat\ell_d})$ is a normed algebra, $\|\cdot\|_{\hat\ell_d}$ is sub-multiplicative, that is $\|\tau*\mu\|_{\hat\ell_d} \leq \|\tau\|_{\hat\ell_d} \|\mu\|_{\hat\ell_d}$ for all $\tau, \mu \in \hat\ell_d$ and we have the embeddings $\ell_1 \hookrightarrow \hat\ell_d \hookrightarrow \ell_2.$ 
$\hat \ell_d$ is commutative and unital with unit $e$ and $\|e\|_{\hat \ell_d}=1$.
\item Let $d \in \N$, then $$S_\tau: L_2([0,1]^d)_0 \to L_2([0,1]^d)_0, \sum \limits_{J \in \Zdpos} a_J \sqrt 2 c_J + b_J \sqrt 2 s_J \mapsto \sum \limits_{J \in \Zdpos} (\tau *_d a)_J \sqrt 2 c_J + (\tau *_d b)_J \sqrt 2 s_J$$ is a well-defined and bounded linear operator if, and only if, $\tau \in \hat\ell_d$ holds, in which case $\|S_\tau\|=\|\tau\|_{\hat\ell_d}$ holds true. $S_\tau$ is bijective if, and only if, $\tau_1 \neq 0$ and $\tau^{-1} \in \hat\ell_d$ holds, where $\tau^{-1}$ is the Dirichlet inverse, in which case $S_\tau^{-1}=S_{\tau^{-1}}$ holds true.
\item Given $\tau \in \ell_2$ define $C^\tau, S^\tau$ as in \eqref{Funktionentau}, then $C^\tau_J, S^\tau_J \in L_2([0,1]^d)$ for all $d \in \N$ and $J \in \Zdpos.$ Define $\mathcal{Z}^\tau_0:=\mathcal Z^\tau \backslash \{\tau_0\},$
then 
is a Riesz basis of $L_2([0,1]^d)_0$ if, and only if, 
$\tau_1\neq 0$ and $\tau, \tau^{-1} \in \hat\ell_d.$ The optimal Riesz constants in this case are $A=\|\tau^{-1}\|_{\hat\ell_d}^{-2}$ and $B=\|\tau\|_{\hat\ell_d}^2.$
\item The operator $S: \hat\ell_d \to \{S_\tau: \tau \in \hat\ell_d\}, \ \tau \mapsto S_\tau$ is an 
isometric isomorphism and an isomorphism of algebras.
\end{enumerate}
\end{proposition}

\begin{proof}
\begin{enumerate}[label=(\roman*)]
\item Let $\tau, \lambda \in \hat\ell_d$ and choose $c(\tau), c(\lambda)$ as in the definition of $
\hat\ell_d.$ For all $a \in \ell_2$ it holds 
$$\|(\tau+\lambda)*_da\|_{\ell_2}=\|\tau*_da+\lambda*_da\|_{\ell_2} \leq \|\tau*_da\|_{\ell_2}+\|\lambda*_da\|_{\ell_2} \leq (c(\tau)+c(\lambda))\|a\|_{\ell_2},$$ 
whence $\tau+\lambda \in \hat\ell_d$ and taking the infimum yields $\|\tau + \lambda\|_{\hat\ell_d}\leq \|\tau\|_{\hat\ell_d}+\|\lambda\|_{\hat\ell_d}.$ Scalar multiples and absolute homogeneity can be addressed similarly. From $\|\tau\|_{\hat\ell_d}=0$ it follows $0 \leq \|\tau*a\|_{\ell_2} \leq 0 \|a\|_{\ell_2}=0$ and $\|\tau\|_{\ell_2}=\|\tau*a\|_{\ell_2}=0$, which implies $\tau*_da=0$ for all $a \in \ell_2(\Zdpos)$ and shows $\tau=0$ by \ref{dDirichlet}(iv). All remaining properties of vector spaces, norms and the multiplication are immediately seen to hold true. 

Let $\tau, \mu \in \hat\ell_d$ and $a \in \ell_2(\Zdpos),$ then $\|(\tau*\mu)*_da\|_{\ell_2}=\|\tau*_d(\mu*_da)\|_{\ell_2} \leq \|\tau\|_{\hat\ell_d} \|\mu*_da\|_{\ell_2} \leq \|\tau\|_{\hat\ell_d} \|\mu\|_{\hat\ell_d} \|a\|_{\ell_2}$ shows $\tau*\mu \in \hat\ell_d$ and $\|\tau*\mu\|_{\hat\ell_d} \leq \|\tau\|_{\hat\ell_d} \|\mu\|_{\hat\ell_d}.$ The embedding $\ell_1 \hookrightarrow \hat\ell_d$ follows from Theorem \ref{Ttau}(i) and Proposition \ref{Coeff}. The embedding $\hat\ell_d \hookrightarrow \ell_2$ follows from $\|\tau\|_{\ell_2} = \|\tau*_da\|_{\ell_2} \leq \|\tau\|_{\hat\ell_d} \|a\|_{\ell_2}=\|\tau\|_{\hat\ell_d},$ where $a_{(1,...,1)}=1$ and $a_J=0$ for all $J \in \Zdpos \backslash \{(1,...,1)\}$.
\item The first part follows immediately from the definition of $\hat\ell_d, S_\tau$ and $\mathcal Z^e \backslash \{1\}$ being an orthonormal basis of $L_2([0,1]^d)_0.$ If $\tau_1 \neq 0$ and $\tau^{-1} \in \hat\ell_d,$ then clearly $S_\tau^{-1}=S_{\tau^{-1}}.$ Let $S_\tau$ be bijective, then one can proceed as in Lemma \ref{SurjTtau} to see $\tau_1 \neq 0.$ Let $g=\sum \limits_{J\in \Zdpos} a_J \sqrt 2 c_J + b_J \sqrt 2 s_J$ be any element of $L_2([0,1]^d)_0$.
Then the coefficients form a square-summable sequence. Because $S_\tau$ is surjective, there is another element $f=\sum \limits_{J \in \Zdpos} \alpha_J \sqrt 2 c_J + \beta_J \sqrt 2 s_J,$ where the coefficients again form a square-summable sequence, such that $S_\tau f=g.$ This shows $\tau*_d\alpha=a$ and $\tau*_d\beta=b,$ from which $\alpha=\tau^{-1}*_da$ and $\beta=\tau^{-1}*_db$ follow. In particular, because any square-summable sequence can be chosen as the coefficients of $g,$ this shows $\tau^{-1}*_da \in \ell_2(\Zdpos)$ for all $a \in \ell_2(\Zdpos)$ 
and that the inverse is given by $$S_\tau^{-1}: L_2([0,1]^d)_0 \to L_2([0,1]^d)_0, \quad \sum \limits_{J \in \Zdpos} a_J \sqrt 2 c_J + b_J \sqrt 2 s_J \mapsto \sum \limits_{J \in \Zdpos} (\tau^{-1} *_d a)_J \sqrt 2 c_J + (\tau^{-1} *_d b)_J \sqrt 2 s_J,$$ which means $S_\tau^{-1}=S_{\tau^{-1}}.$ By the open mapping theorem, $S_\tau^{-1}=S_{\tau^{-1}}$ is bounded, implying $\tau^{-1} \in \hat\ell_d.$ 

\item The first part follows from $
\mathcal Z^e_0:=\mathcal Z^e \backslash \{1\}$ being an orthonormal basis of the Hilbert space $L_2([0,1]^d)_0,$
since the coefficients are in $\ell_2.$ Further, $S_\tau$ maps the elements of $
\mathcal Z^e_0$ to those of $\mathcal Z^\tau_0.$ 
By \cite[Ch.3]{Basen} it follows that $\mathcal Z^\tau_0$
is a Riesz basis of $L_2([0,1]^d)_0$ if, and only if, $S_\tau$ is bijective and $S_\tau, S_\tau^{-1}=S_{\tau^{-1}}$ are bounded. 
The rest follows from (ii), 
in particular, that the optimal Riesz constants in this case are $B=\|S_\tau\|^2=\|\tau\|^2_{\hat\ell_d}$ and $A=\|S_\tau^{-1}\|^{-2}=\|\tau^{-1}\|^{-2}_{\hat\ell_d}.$
\item $S$ being an isomorphism of algebras follows from $*_d$ being pseudo-associative, see 
Lemma \ref{dDirichlet}(ii), and bilinear. Clearly, $S$ is surjective and an isometric isomorphism by (ii), from which injectivity and consequently bijectivity follow. 
\end{enumerate}
\end{proof}

To regain a Riesz basis of $L_2([0,1]^d)$ from a Riesz basis of $L_2([0,1]^d)_0$ one can add any non-zero constant map to the system. If one chooses $\|\tau\|_{\ell_2}$ every element has the same norm. 

The previously developed results were always independent of the dimension $d.$ When investigating the sequence spaces $\hat \ell_d$ it turns out that $\hat \ell_d=\hat \ell_1$ holds for all $d \in \N.$

\begin{lemma}[{$\boldsymbol{\hat \ell_d=\hat \ell_1}$}]\label{ellhatd}
   It holds $\hat \ell_d = \hat \ell_1$ and $\|\cdot\|_{\hat \ell_1} = \|\cdot\|_{\hat \ell_d}$ for all $d \in \N.$
\end{lemma}

\begin{proof}
Let $P_d:=\{J\in\Zdpos:\gcd(J_1,\dots,J_d)=1\}$ denote the set of primitive vectors in $\Zdpos$. Every $K\in\Zdpos$ can be written uniquely as $K=nJ$ with $n\in\N$ and $J\in P_d$, namely $n=\gcd(K_1,\dots,K_d)$ and $J=\frac1nK$. Moreover, if $J\in P_d$, $k,n\in\N$ and $M\in\Zdpos$ satisfy $kM=nJ$, 
then $M=\frac nkJ$ by Lemma \ref{kJ=N}, since $(nJ)_0=n$ for $J\in P_d$. 
Let $\tau \in \hat \ell_d$ and $a \in \ell_2.$ Pick $J=(1,\dots,1)\in P_d$ and define $b:\Zdpos \to \R$ by $b_{kJ}:=a_k$ for all $k \in \N$ and $b_M:=0$ else. Then $b\in\ell_2(\Zdpos)$ with $\|b\|_{\ell_2(\Zdpos)}=\|a\|_{\ell_2}$, and by the above
$$(\tau*a)_n=\sum \limits_{k|n} \tau_k a_{\frac n k} = \sum \limits_{k|n} \tau_k b_{\frac n k J} = (\tau *_d b)_{nJ} \quad\text{for all } n\in\N.$$
Since $b$ vanishes outside $\N J$ and $kM\in\N J$ for all $M\in\N J$, also $\tau*_db$ vanishes outside $\N J$. Whence
$$\|\tau*a\|_{\ell_2}=\|\tau*_db\|_{\ell_2(\Zdpos)} \leq \|\tau\|_{\hat \ell_d} \|b\|_{\ell_2(\Zdpos)}=\|\tau\|_{\hat \ell_d} \|a\|_{\ell_2},$$
which shows $\hat \ell_d \subset \hat \ell_1$ and $\|\cdot\|_{\hat \ell_1} \leq \|\cdot\|_{\hat \ell_d}.$

Now let $\tau \in \hat \ell_1$ and $a \in \ell_2(\Zdpos).$ For $J\in P_d$ define $a^J\in\ell_2$ by $a^J_n:=a_{nJ}$, $n\in\N$. By the above, $(\tau*_da)_{nJ}=\sum_{k|n}\tau_ka_{\frac nkJ}=(\tau*a^J)_n$ for all $n\in\N$ and $J\in P_d$. Using the unique decomposition $K=nJ$ we obtain
$$\|\tau*_da\|_{\ell_2(\Zdpos)}^2 = \sum_{J \in P_d} \sum_{n=1}^\infty (\tau*_da)_{nJ}^2 = \sum_{J \in P_d} \|\tau*a^J\|_{\ell_2}^2 \leq \|\tau\|_{\hat \ell_1}^2 \sum_{J \in P_d} \|a^J\|_{\ell_2}^2 = \|\tau\|_{\hat \ell_1}^2 \sum_{J \in P_d}\sum_{n=1}^\infty a_{nJ}^2 = \|\tau\|_{\hat \ell_1}^2 \|a\|_{\ell_2(\Zdpos)}^2.$$
Taking square roots gives $\|\tau*_da\|_{\ell_2(\Zdpos)} \leq \|\tau\|_{\hat \ell_1} \|a\|_{\ell_2(\Zdpos)}$, which shows $\hat \ell_1 \subset \hat \ell_d$ and $\|\cdot\|_{\hat \ell_d} \leq \|\cdot\|_{\hat \ell_1}.$
\end{proof}

The characterisation of the space $\hat \ell_1$ is equivalent to a related question that has been thoroughly studied in \cite{Seip}. 



\begin{remark}[\textbf{Comparison with \cite{Seip}}]\label{CompHLS}
For $\tau\in\ell_1$ let $F_\tau(s):=\sum_{n=1}^\infty\tau_n n^{-s}$. By definition, $\hat\ell_1$ consists of the coefficient sequences of the multipliers of the Hilbert space $\mathcal H^2=\{\sum_n a_nn^{-s}:a\in\ell_2\}$ of Dirichlet series, and $\|\tau\|_{\hat\ell_1}$ is the multiplier norm. By \cite[Thm.~3.1]{Seip} the multipliers of $\mathcal H^2$ are exactly the Dirichlet series that converge to bounded analytic functions in $\operatorname{Re}s>0$, and
\[
\|\tau\|_{\hat\ell_1}=\sup_{\operatorname{Re}s>0}|F_\tau(s)|.
\]
Since $F_{\tau^{-1}}=1/F_\tau$, Proposition~\ref{ellhat}(iii) yields the optimal Riesz constants
\[
A=\inf_{\operatorname{Re}s>0}|F_\tau(s)|^2,\qquad B=\sup_{\operatorname{Re}s>0}|F_\tau(s)|^2
\]
of $\mathcal Z^\tau_0$ in $L_2([0,1]^d)_0$, independently of $d$, by Lemma \ref{ellhatd}.

For $d=1$ the restriction to $[0,\frac12]$ followed by the substitution $x=\frac y2$ maps $\overline{\operatorname{span}}\{s_k:k\in\N\}$ isometrically onto $L_2(0,1)$ and $S^\tau_k$ onto $\varphi(ky)$ with $\varphi(y)=\sqrt2\sum_{n}\tau_n\sin(\pi ny)$. Thus, for the sine part, Proposition~\ref{ellhat}(iii) reduces to the Riesz basis criterion of \cite{Seip}. The new aspects of our setting, besides a different method of proof, are the simultaneous treatment of the cosine and the sine part, the multivariate systems indexed by $\Zdpos$, the spaces $L_q([0,1]^d)$ with $q\neq2$, as well as the generalisation to mixed systems in 
the following Subsection \ref{subsect-mixed}.
\end{remark}

\subsection{Mixed systems}\rev{\label{subsect-mixed}}

Until now we only studied function systems for which the coefficients of the cosine series $C^\tau$ and the coefficients of the sine series $S^\tau$ were perfectly equal. 
In this subsection the cosine and the sine part of the system are instead generated by two different sequences $\tau$ and $\lambda$. We first prove the basis properties of mixed systems (Theorem \ref{MixedZystems}). Afterwards we introduce the reflection operator $R$ and the operators $T_\tau^R$ and show that $T_\tau+T_\lambda^R$ maps the trigonometric system $\mathcal Z^e$ to such a mixed system (Lemma \ref{MixOp}). This allows the generalisation of the algebraic behaviour that was seen in Theorem \ref{AlgTtau} and the invertibility criterion of Theorem \ref{InvTtau} to these operators (Lemma  \ref{MixAlg}).

\begin{theorem}[\textbf{Mixed systems}] \label{MixedZystems}
    Let $q\in(1,\infty), d \in \N, k_0 \neq 0$ and $\tau, \lambda \in \ell_1$ such that $\mathcal Z^\tau, \ \mathcal Z^\lambda$ are Schauder bases of $L_q([0,1]^d).$ Then, the system $$\{k_0\}\cup \{C_J^\tau, S_J^\lambda: J \in \Zdpos\}$$ is a Schauder basis of $L_q([0,1]^d)$ as well.  
    
    If $q=2,$ it is a Riesz basis, if $\mathcal Z^\tau$ and $\mathcal Z^\lambda$ are Riesz bases with constants $A_\tau \leq B_\tau$ and $A_\lambda \leq B_\lambda,$ where $$A=\min\{k_0^2,A_\tau, A_\lambda\}, \quad B=\max\{k_0^2,B_\tau, B_\lambda\}.$$
\end{theorem}
\begin{proof}
We will prove this theorem in four steps: First we will show that any $g \in L_q([0,1]^d)$ can uniquely be represented as $g=g_0+g_++g_-,$ where $g_0$ is a constant, $g_+$ is even with vanishing mean and $g_-$ is odd (with vanishing mean). Then, we will use this to find representations with respect to the given system and afterwards show uniqueness. Lastly, we will estimate the Riesz constants. Let $g \in L_q([0,1]^d).$

\textbf{Step 1:} We set $g_0:=\int_{[0,1]^d} g(x) \ \mathrm d x.$ Then, $g= g_0 + g_+ + g_-,$ where $$g_+(x):=\frac{g(x)+g(-x)}2-g_0, \quad g_-(x) := \frac{g(x)-g(-x)}2,$$ is a suitable representation. We will prove that it is the only one.

Assume there would be another representation $g=g_0'+g_+'+g_-'.$ Integrating $g$ over $[0,1]^d$ shows $g_0'=g_0.$ Now from $g_0+g_++g_-=g=g_0'+g_+'+g_-'$ we get $g_+-g_+'=g_-'-g_-.$ The left hand side is an even function, while the right hand side is an odd function. The only such function is the zero function, since any odd and even function $h$ satisfies $-h(x)=h(-x)=h(x)$ and consequently $h(x)=0$ for almost every $x \in [0,1]^d.$ Whence, $g_+-g_+'=0=g_-'-g_-$ and $g_+=g_+', g_-=g_-'$ follow.

\textbf{Step 2:}
Since $\mathcal Z^\tau$ is a Schauder basis of $L_q([0,1]^d)$, we can find a representation of $g_+$ with respect to {$\mathcal Z^\tau,$} that is, $a_J, b_J, \gamma_0 \in \R$ such that $$g_+=\gamma_0+\sum \limits_{J \in \Zdpos} a_J C_J^\tau + \sum \limits_{J \in \Zdpos} b_J S_J^\tau.$$ Integrating over $[0,1]^d$ shows $\gamma_0=0.$ Furthermore, since $g_+$ is an even function, $$g_+(x)=\frac{g_+(x)+g_+(-x)}{2}=\sum \limits_{J \in \Zdpos} a_J \frac{C_J^\tau(x)+C_J^\tau(-x)}2 + \sum \limits_{J \in \Zdpos} b_J \frac{S_J^\tau(x)+S_J^\tau(-x)}2 = \sum \limits_{J \in \Zdpos} a_J C_J^\tau$$ follows. Similarly, $g_-$ can be expressed with respect to $\mathcal Z^\lambda$ as $g_-(x)=\sum \limits_{J \in \Zdpos} a'_J S_J^\lambda$ giving $$g=g_0+g_++g_-=\frac{\int_{[0,1]^d} g(x) \ \mathrm d x} {k_0} \cdot k_0 + \sum \limits_{J \in \Zdpos} a_J C_J^\tau + \sum \limits_{J \in \Zdpos} a'_J S_J^\lambda.$$ 
 
\textbf{Step 3:}
Let $$g=b_0 \cdot k_0+ \sum \limits_{J \in \Zdpos} b_J C_J^\tau + \sum \limits_{J \in \Zdpos} b'_J S_J^\lambda$$ be any representation of $g$ with respect to the system. The coefficients $b_J, b_J'$ are uniquely determined by $g$ because of Proposition \ref{Coeff}(ii). Similarly $b_0$ is uniquely determined by $b_0 \cdot k_0 = \int_{[0,1]^d} g(x) \mathrm d x,$ since $k_0\neq 0.$

\textbf{Step 4:} Due to orthogonality we can estimate \begin{align*}
    \left\|a_0 \cdot k_0 + \sum \limits_{J \in \Zdpos} a_J C_J^\tau + \sum \limits_{J \in \Zdpos} a'_J S_J^\lambda\right\|^2 &= a_0^2k_0^2 + \left\|\sum \limits_{J \in \Zdpos} a_J C_J^\tau\right\|^2 + \left\| \sum \limits_{J \in \Zdpos} a'_J S_J^\lambda\right\|^2 \\ &\leq k_0^2 \cdot a_0^2 + B_\tau \cdot \sum \limits_{J \in \Zdpos}a_J^2 + B_\lambda \cdot \sum \limits_{J \in \Zdpos}(a_J')^2 \\ &\leq \max\{k_0^2,B_\tau, B_\lambda\} \left(a_0^2+\sum \limits_{J \in \Zdpos}(a_J^2+(a_J')^2)\right) \end{align*} and similarly $$\left\|a_0 \cdot k_0 + \sum \limits_{J \in \Zdpos} a_J C_J^\tau + \sum \limits_{J \in \Zdpos} a'_J S_J^\lambda\right\|^2 \geq \\ \min\{k_0^2,A_\tau, A_\lambda\} \left(a_0^2+\sum \limits_{J \in \Zdpos}(a_J^2+(a_J')^2)\right).$$
\end{proof}

These mixed systems can again be obtained from $\mathcal Z^e$ using a bounded operator $T.$ We start with a definition.

\begin{definition}[\textbf{The operators $R$ and $T_\tau^R$}]
    Let $q \in (1,\infty)$ and $d \in \N.$ We define the reflection map $$R:L_q([0,1]^d) \to L_q([0,1]^d), \ Rg(x) := g(-x).$$ Further, given $\tau \in \ell_1$ we set $$T_\tau^R: L_q([0,1]^d) \to L_q([0,1]^d), \quad T_\tau^Rg(x):= (T_\tau \circ R) g(x)= \sum \limits_{k=1}^\infty \tau_k g(-kx).$$
\end{definition}

The following facts are easily seen: Firstly, $R$ is its own inverse and an isometry. Secondly, for any $J \in \Zdpos$ it holds $Rc_J=c_J$ and $Rs_J=-s_J$ as well as $R1=1.$ Lastly, for any $\tau \in \ell_1$ the operators $R$ and $T_\tau$ commute. 

\begin{lemma}[\textbf{Mixed system operator}] \label{MixOp}
    Let $q \in (1,\infty), d \in \N, J \in \Zdpos$ and $\tau, \lambda \in \ell_1.$ The operator $T=T_{\frac{\tau+\lambda} 2} + T_{\frac{\tau-\lambda} 2}^R \in \mathcal L(L_q([0,1]^d))$ satisfies $$T 1 = \tau_0, \quad T (\sqrt 2 c_J)=C_J^\tau \quad \text{and} \quad T(\sqrt 2 s_J)=S^\lambda_J.$$
\end{lemma}

\begin{proof}
    Because $\ell_1$ is a vector space it holds $\frac{\tau \pm \lambda}{2} \in \ell_1.$ By Theorem \ref{Ttau}(i), $T$ is linear and bounded. It holds \begin{align*}
    T(1) &= \sum \limits_{k=1}^\infty \frac{\tau_k + \lambda_k}2 + \sum \limits_{k=1}^\infty \frac{\tau_k-\lambda_k}{2} = \sum \limits_{k=1}^\infty \tau_k = \tau_0, \\ T(\sqrt 2 c_J) &= \sqrt 2 \sum \limits_{k=1}^\infty \frac{\tau_k + \lambda_k}2c_{kJ} + \sqrt 2 \sum \limits_{k=1}^\infty \frac{\tau_k - \lambda_k}2c_{kJ}=\sqrt 2 \sum \limits_{k=1}^\infty \tau_kc_{kJ}=
    C^\tau_J, \\
    T(\sqrt 2 s_J) &= \sqrt 2\sum \limits_{k=1}^\infty \frac{\tau_k + \lambda_k}2s_{kJ} - \sqrt 2\sum \limits_{k=1}^\infty \frac{\tau_k - \lambda_k}2s_{kJ} = \sqrt 2 \sum \limits_{k=1}^\infty \lambda_k s_{kJ}=S^\lambda_J.
\end{align*}
\end{proof}

Lemma \ref{MixOp} realises the mixed systems as images of $\mathcal Z^e$ under operators of the form $T_\tau+T_\lambda^R$. In case $q=2$ one can show, that these operators are again homeomorphisms, if the corresponding mixed system forms a Riesz basis of $L_2([0,1]^d),$ see \cite[Ch. 3.6]{Basen} As in the case of the operators $T_\tau$, these operators form an algebra, and their inverses can be computed explicitly from Dirichlet inverses of sequences. This is the content of the following lemma.

\begin{lemma}[Algebraic behaviour of $T_\tau+T_\lambda^R$] \label{MixAlg}
    Let $q \in (1,\infty), d \in \N, 
    \alpha \in \R$ and $\tau, \lambda, \kappa,\nu \in \ell_1.$ In $\mathcal L(L_q([0,1]^d))$ it holds
    \begin{enumerate}[label=(\roman*)]
        \item$
        \alpha(T_\tau+T_\lambda^R)+(T_\kappa+T_\nu^R)=T_{
        \alpha \tau+\kappa} + T_{
        \alpha \lambda+\nu}^R$
        \item $(T_\kappa+T_\nu^R) \circ(T_\tau + T_\lambda^R)=(T_\tau + T_\lambda^R)\circ(T_\kappa+T_\nu^R)=T_{\tau*\kappa+\lambda*\nu}+T^R_{\tau*\nu+\lambda*\kappa}$
        \item $(T_\tau+T_\lambda^R)^{-1}=T_{(\tau^{*2}-\lambda^{*2})^{-1}*\tau}+T_{-(\tau^{*2}-\lambda^{*2})^{-1}*\lambda}^R$ as long as $|\tau_1|\neq |\lambda_1|$ and $(\tau^{*2}-\lambda^{*2})^{-1}\in \ell_1.$
    \end{enumerate}    
\end{lemma}

\begin{proof}
    Using Theorem \ref{AlgTtau} parts (i) and (ii) are easily seen, we only prove (iii).

    The sequence $\tau^{*2}-\lambda^{*2}$ admits a Dirichlet inverse, since $(\tau^{*2}-\lambda^{*2})_1=(\tau*\tau)_1-(\lambda*\lambda)_1=
    \tau^2_1-\lambda^2_1 \neq 0$ holds because of the assumption $|\tau_1| \neq |\lambda_1|$. Because the inverse was further assumed to be summable, the operator $T_{(\tau^{*2}-\lambda^{*2})^{-1}*\tau}+T_{-(\tau^{*2}-\lambda^{*2})^{-1}*\lambda}^R$ is well-defined. Lastly, using (ii) we check \begin{align*}
        (T_\tau+T_\lambda^R)(T_{(\tau^{*2}-\lambda^{*2})^{-1}*\tau}+T_{-(\tau^{*2}-\lambda^{*2})^{-1}*\lambda}^R)=T_{(\tau^{*2}-\lambda^{*2})*(\tau^{*2}-\lambda^{*2})^{-1}}+T^R_{(\lambda*\tau-\tau*\lambda)*(\tau^{*2}-\lambda^{*2})^{-1}}=T_e+T_0^R=\text{Id}.
    \end{align*}
\end{proof}

\subsection{Examples} \label{subsect-ex}
We now give two examples of sequences $\tau\in\ell_1$ to which Theorem~\ref{Basiscrit} applies and for which the series defining $C^\tau$ and $S^\tau$ can be summed explicitly. This phenomenon is further developed in the subsequent Remark \ref{Power series}.

\begin{example}[\textbf{Applications of Theorem~\ref{Basiscrit}}] \label{Example}
\begin{enumerate}[label=(\roman*)]
    \item \emph{The geometric sequence.} Let $r\in(-1,1)$ and $\tau_k:=r^k.$ Then $\sum_{n=1}^\infty |r|^n=\frac{|r|}{1-|r|}$ converges, whence $\tau\in\ell_1$. We check whether $2|\tau_1|>\|\tau\|_{\ell_1}$ holds. For $r=0$ this is not the case, so let $r\neq0.$ Then $$2|r|>\frac{|r|}{1-|r|}\ \text{ is equivalent to }\ 1-|r|>\frac12, \ \text{ that is, to }\ |r|<\frac12.$$ Whence Theorem~\ref{Basiscrit}(i) applies for every $r\in(-\frac12,\frac12)\setminus\{0\}$, and for $q=2$ it yields the Riesz constants $B=\big(\frac{|r|}{1-|r|}\big)^2$ and $A\geq\big(|r|\frac{1-2|r|}{1-|r|}\big)^2.$

    The functions $C^\tau, S^\tau$ can be written in closed form. For every $r\in(-1,1)$ and $x\in\R$ we have $|re^{2\pi ix}|=|r|<1$, and therefore
    \begin{align*}
        &\sum_{k=1}^\infty r^k\cos(2\pi kx)+i\sum_{k=1}^\infty r^k\sin(2\pi kx)
        =\sum_{k=1}^\infty \big(re^{2\pi ix}\big)^k=\frac{re^{2\pi ix}}{1-re^{2\pi ix}}
        =\frac{re^{2\pi ix}\big(1-re^{-2\pi ix}\big)}{\big(1-re^{2\pi ix}\big)\big(1-re^{-2\pi ix}\big)}\\
        &\qquad=\frac{re^{2\pi ix}-r^2}{1+r^2-2r\cos(2\pi x)}
        =\frac{\big(r\cos(2\pi x)-r^2\big)+ir\sin(2\pi x)}{1+r^2-2r\cos(2\pi x)}.
    \end{align*}
    Taking real and imaginary parts, we obtain $$C^\tau(x)=\sqrt2\,\frac{r\cos(2\pi x)-r^2}{1+r^2-2r\cos(2\pi x)}\quad\text{and}\quad S^\tau(x)=\sqrt2\,\frac{r\sin(2\pi x)}{1+r^2-2r\cos(2\pi x)}.$$
    \begin{figure}[h]
    \begin{subfigure}[c]{0.5\textwidth}
    \centering
    \includegraphics[width=0.7\textwidth]{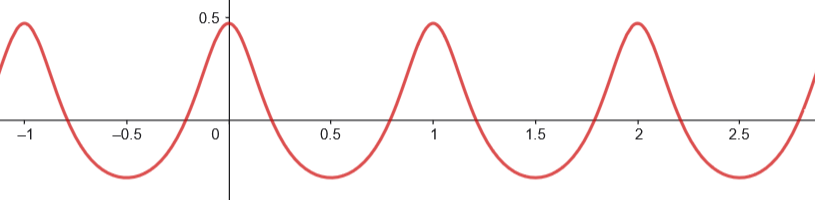}
    \subcaption{$C^\tau$}
    \end{subfigure}
    \begin{subfigure}[c]{0.5\textwidth}
    \centering
    \includegraphics[width=0.7\textwidth]{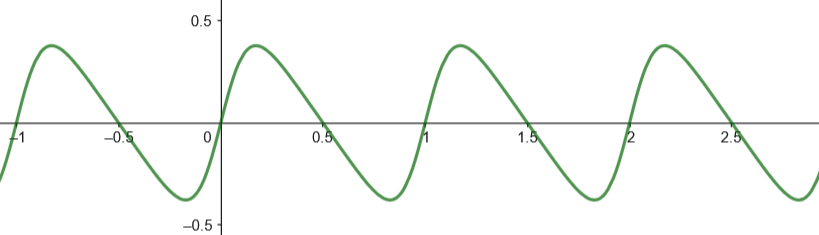}
    \subcaption{$S^\tau$}
    \end{subfigure}
    \caption{Plots of $C^\tau$ and $S^\tau$ in Example~\ref{Example}(i) with $r=\frac14$}
    \end{figure}
    \item \emph{The exponential series.} Let $\tau_k:=\frac1{k!}.$ Then $\tau\in\ell_1$ and $\|\tau\|_{\ell_1}=e-1<2=2|\tau_1|$, so Theorem~\ref{Basiscrit}(i) applies, and for $q=2$ the Riesz constants are $B=(e-1)^2$ and $A\geq(3-e)^2.$ In $\C$ we compute
    \begin{align*}
        &\sum_{k=1}^\infty\frac{\cos(2\pi kx)}{k!}+i\sum_{k=1}^\infty\frac{\sin(2\pi kx)}{k!}
        =\sum_{k=1}^\infty\frac{\big(e^{2\pi ix}\big)^k}{k!}=\exp\big(\cos(2\pi x)+i\sin(2\pi x)\big)-1\\
        &\qquad=e^{\cos(2\pi x)}\big(\cos(\sin(2\pi x))+i\sin(\sin(2\pi x))\big)-1,
    \end{align*}
    which shows $$C^\tau(x)=\sqrt2\,e^{\cos(2\pi x)}\cos(\sin(2\pi x))-\sqrt2\quad\text{and}\quad S^\tau(x)=\sqrt2\,e^{\cos(2\pi x)}\sin(\sin(2\pi x)).$$
    \begin{figure}[h]
    \begin{subfigure}[c]{0.5\textwidth}
    \centering
    \includegraphics[width=0.6\textwidth]{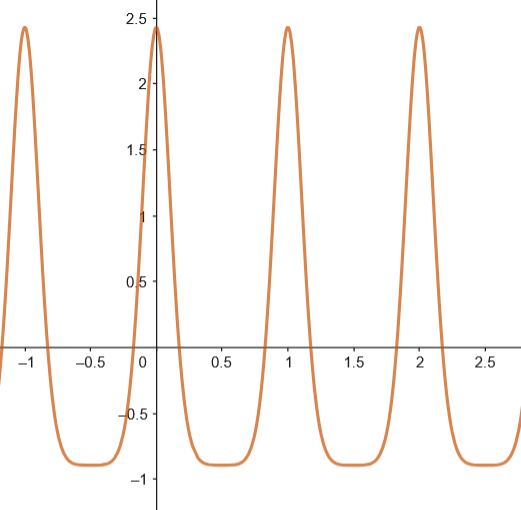}
    \subcaption{$C^\tau$}
    \end{subfigure}
    \begin{subfigure}[c]{0.5\textwidth}
    \centering
    \includegraphics[width=0.6\textwidth]{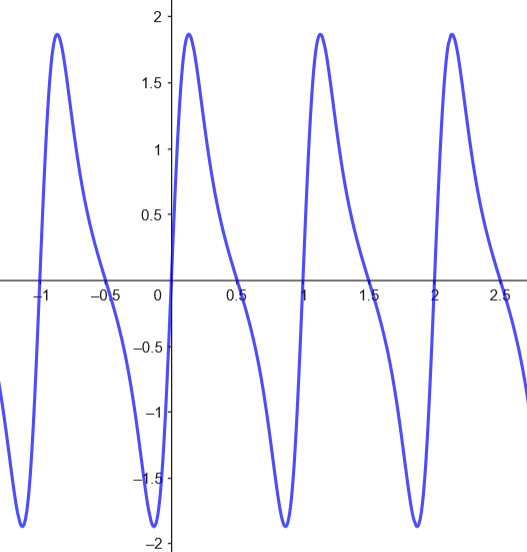}
    \subcaption{$S^\tau$}
    \end{subfigure}
    \caption{Plots of $C^\tau$ and $S^\tau$ in Example~\ref{Example}(ii)}
    \end{figure}
\end{enumerate}
\end{example}

\begin{remark}{\textbf{(Holomorphic interpretation).}} \label{Power series} \hypertarget{Power series} Let $\tau \in \ell_1.$ In order to obtain closed expressions for $C^\tau, S^\tau$ we make use of Euler's formula and compute $C^\tau(x) + i S^\tau(x)$ using Taylor series. More generally one may study $\mathcal Z^\tau$ in a complex framework: Consider the power series $$f(z):=\sum \limits_{n=1}^\infty \tau_n z^n,$$ which converges for all $z \in \C$ where $|z| \leq 1,$ thus showing that the radius of convergence $\rho$ is at least 1. This allows the evaluation $f(\exp(ix))$ for all $x \in \R$ since  $|\exp(ix)|=1,$ and leads to the formulae $$C^\tau(x) = \sqrt{2}\text{Re}(f(\exp(ix))=\sqrt{2}\frac{f(\exp(ix))+\overline{f(\exp(ix)}} 2 =\sqrt{2} \frac{f(\exp(ix))+f(\exp(-ix))}2,$$ as well as $$S^\tau(x) = \sqrt{2}\text{Im}(f(\exp(ix))=\sqrt{2}\frac{f(\exp(ix))-\overline{f(\exp(ix)}} {2i} = \sqrt{2}\frac{f(\exp(ix))-f(\exp(-ix))}{2i}.$$

Whenever we can find closed expressions for $f,$ we are also able to find analytical expressions for $C^\tau$ and $S^\tau.$ 
If the radius of convergence satisfies $\rho>1,$ then $f$ is holomorphic, whence smooth, on $B_\rho(0).$ Since for all $x \in \R$ we have $|\exp(ix)|=1<\rho,$ this shows that $C^\tau$ and $S^\tau$ are even smooth in this case. It is possible that some other insides can be gathered by the means of complex analysis, but we will not look into that.

    Such a power series $f$ can be interpreted as the complex analogue to the operator $T_\tau.$ If additionally $\lambda \in \ell_1$ we can also interpret the operator $T$ from Lemma \ref{MixOp} in terms of holomorphic functions. The reflection operator $R$ behaves similar to the complex conjugation, which 
    however is not an analytic function. Therefore, we need to find a better representation. 
    Since we are only interested in evaluating $\exp(ix)$ the conjugate of which is $\exp(-ix)=\frac{1}{\exp(ix)},$ we will identify $R$ with taking the multiplicative inverse, which is holomorphic on $\C \backslash \{0\}.$ 
    The analytic analogue to the operator $$T=T_\frac{\tau+\lambda}{2}+T_\frac{\tau-\lambda}{2}^R, \quad (Tg)(x) := \sum \limits_{k=1}^\infty \frac{\tau_k+\lambda_k}{2} g(kx) + \sum \limits_{k=1}^\infty \frac{\tau_k-\lambda_k}{2}g(-kx)$$ then takes the form $$f(z):=\sum \limits_{k=1}^\infty \frac{\tau_k+\lambda_k}{2}z^k + \sum \limits_{k=1}^\infty \frac{\tau_k-\lambda_k}{2}z^{-k}.$$ Since $\frac{\tau\pm\lambda}{2} \in \ell_1$ the 
    analytic part converges at least on $\{z \in \C: |z| \leq 1\}$ 
    whereas the principal part converges at least on $\{z \in \C: |z|\geq 1\}.$ Thus, we may evaluate $f(\exp(ix))$ for all $x \in \R$ since $|\exp(ix)|=1$ holds. 
      In order to use the theory from complex analysis we   need to ensure better convergence once more. Let $\rho_g$ be the radius of convergence of the power series $g(z)=\sum \limits_{k=1}^\infty \frac{\tau_k+\lambda_k}{2}z^k$ and $\rho_h$ the radius of convergence of the power series $\sum \limits_{k=1}^\infty \frac{\tau_k-\lambda_k}{2}z^k,$ then from the summability of the coefficients we get $\rho_g, \ \rho_h \geq 1.$ 
      In case both are strictly greater than $1$, the Laurent series
$f(z)=g(z)+h(1/z)$ with $h(w):=\sum_{k\ge1}\frac{\tau_k-\lambda_k}{2}w^k$
converges on the open annulus $\{z\in\C:\rho_h^{-1}<|z|<\rho_g\}$ and defines a holomorphic function there. 
This annulus contains the unit circle and whence $\{\exp(ix):x\in\R\}$.
For details about the convergence and holomorphity of power series and the expansion into Taylor series see for instance \cite[I §7 and §8, III §4 and §5]{Funktionentheorie} and for details about convergence and holomorphity of Laurent series, see for instance \cite[VI§1]{Funktionentheorie}.
 Lastly, 
 let us point out again that 
 in \cite{Seip} a very similar framework was studied where a connection to Dirichlet series was established, which deepens the connection to holomorphic functions, see Remark \ref{CompHLS}. It is highly likely that a similar procedure can be implemented in our setting.
\end{remark}

\section{Application to the ReLU-inspired system $\mathcal R_d$}\label{sect-app}

In this section we apply the general theory to the system $\mathcal R_d$ from the introduction. Our results are multivariate analogues of the one-dimensional results of \cite{4}. We first collect some elementary properties of $\mathcal C$ and $\mathcal S$, in particular which identities for sine and cosine carry over to them. We then compute the scalar products of the dilated functions $\mathcal C_J$ and $\mathcal S_J$ for $J \in \Zdpos$ (Theorem \ref{SKP2}, extending \cite[Lem. 2.1]{4}) and show that the truncated Gram matrices of the cosine part and of the sine part have the same spectra, which extends \cite[Thm. 2.2]{4} to arbitrary dimensions $d \in \N$. Finally, we write $\mathcal R_d$ as a mixed system in the sense of Theorem \ref{MixedZystems}. This is used to show that $\mathcal R_d$ is a Schauder basis of $L_q([0,1]^d)$ and, for $q=2$, a Riesz basis with the improved lower Riesz constant $\frac 2 3$ (Theorem \ref{Rdbasis}). 

We start with a positive result on the generalisation of trigonometric identities.




\begin{lemma}[\textbf{Double Argument Theorem for $\mathcal C$}] \label{DWSC}
    For all $x \in \R$ it holds that $$\mathcal C(2x)= \mathcal C(x)^2 - \mathcal S(x)^2.$$
\end{lemma}
\begin{proof}
Since both maps are 1-periodic, it is enough to show equality on [0,1). It holds
$$\mathcal C(x) = \begin{cases}
    1-4x, &\text{ if } 0 \leq x \leq \frac 1 2, \\
    4x-3, &\text{ if } \frac 1 2 < x < 1,
\end{cases}$$ 
as well as 
$$\mathcal S(x) = \begin{cases}
    4x, &\text{ if } 0 \leq x \leq \frac 1 4, \\
    2-4x, &\text{ if } \frac 1 4 < x < \frac 3 4, \\
    4x-4, &\text{ if } \frac 3 4 \leq x < 1.
\end{cases}$$ 
We will distinguish the cases $x \in [0, \frac 1 4], (\frac 1 4, \frac 1 2], (\frac 1 2, \frac 3 4), [\frac 3 4, 1).$ This yields \\
\begin{itemize}
    \item $x \in {[0,\frac 1 4]}: \mathcal C(2x) = 1-8x = (1-4x)^2-(4x)^2 = \mathcal C(x)^2-\mathcal S(x)^2,$
    \item $x \in {(\frac 1 4, \frac 1 2]}: \mathcal C(2x)=8x-3 = (1-4x)^2-(2-4x)^2 = \mathcal C(x)^2-\mathcal S(x)^2,$ 
    \item $x \in {(\frac 1 2, \frac 3 4)}: \mathcal C(2x) = \mathcal C(2x-1) = 5-8x = (4x-3)^2-(2-4x)^2 = \mathcal C(x)^2-\mathcal S(x)^2,$ 
    \item $x \in {[\frac 3 4, 1)}: \mathcal C(2x) = \mathcal C(2x-1) = 8x-7 = (4x-3)^2 - (4x-4)^2 = \mathcal C(x)^2-\mathcal S(x)^2,$ 
\end{itemize} 
which completes the proof.
\end{proof}

\noindent This is quite remarkable since $\mathcal C(2x)$ is again a piecewise linear function, whereas $\mathcal C(x)^2 - \mathcal S(x)^2$ a priori might not be. 
Also interesting is that the addition theorem of cosine also holds partially, namely on $[0, \frac 1 4]^2$ as for all $x,y \in [0,\frac 1 4]$ we have the equality $$\mathcal C(x+y)= 1-4(x+y) = (1-4x)(1-4y) - (4x)(4y) = \mathcal C(x) \mathcal C(y) - \mathcal S(x) \mathcal S(y).$$
Sadly for general $x,y \in \R$ this is not the case. Taking $x=\frac{1}{5}, y=\frac{1}{3}$ yields $$\mathcal C\left(\frac 1 5 + \frac 1 3\right) = -\frac {13} {15} \neq - \frac 3 5 = -\frac 1 {15} - \frac 4 5 \cdot \frac 2 3 = \mathcal C\left(\frac 1 5\right) \mathcal C\left(\frac 1 3\right) - \mathcal S\left(\frac 1 5\right) \mathcal S\left(\frac 1 3\right).$$
Similar statements regarding $\mathcal S$ turn out to be false. The analogue to Lemma \ref{DWSC} for $\mathcal S$ would take the form $\mathcal S(2x)=2 \mathcal S(x) \mathcal C(x).$ This cannot hold true, since the expression 
on the right hand side depends quadratically on $x,$ while $\mathcal S(2x)$ is again piecewise linear. Consequently the addition theorem of the sine function cannot be carried over to $\mathcal S.$ Unsurprisingly, the Pythagorean identity $\cos^2+\sin^2=1$ (which is characteristic for the trigonometric functions) does not generalise to $\mathcal C$ and $\mathcal S$ either. However, because sine and cosine are piecewise convex where the inflection points are located at the roots, we can infer $|\mathcal C| \leq |\cos|, |\mathcal S| \leq |\sin|$ and $\mathcal C^2+ \mathcal S^2 \leq 1.$

The scalar products of the functions $\mathcal C_J$ and $\mathcal S_J$ determine the Gram matrices of $\mathcal R_d$ and whence its Riesz constants. For $d=1$ they were computed in \cite[Lem. 2.1]{4}: all scalar products vanish unless the frequencies are oddly co-linear, in which case they are given by an explicit rational number, and the cosine part and the sine part are orthogonal to each other. Moreover, it was shown in \cite[Thm. 2.2]{4} that the truncated Gram matrices of the cosine system and of the sine system are similar, so that their spectra coincide. We now extend both statements to arbitrary dimensions $d$. First we compute the scalar products of $\mathcal C_J$ and $\mathcal S_J$ for $J\in\Zdpos$ by induction over the dimension $d$, without using Fourier analysis (Theorem \ref{SKP2}). Afterwards we show that the truncated Gram matrices of the cosine part and of the sine part of $\mathcal R_d$ are similar (Theorem \ref{Gram}). Recall that $\alpha,\beta\in\Zdpos$ are called oddly co-linear if there exist $p,q\in\N_0$ with $\gcd(2p+1,2q+1)=1$ and $\alpha=\frac{2p+1}{2q+1}\beta$ and that in this case $p$ and $q$ are uniquely determined.

\begin{theorem}[\textbf{Scalar products of $\mathcal C_\alpha, \mathcal S_\alpha$}]\label{SKP2}
    Let $d \in \N, \alpha,\beta \in \Zdpos.$
    \begin{enumerate}[label=(\roman*)]
        \item If $\alpha,\beta$ are oddly co-linear, that is, if there are $p,q \in \N_0$ such that $\alpha = \frac {2p+1} {2q+1} \beta,$ and $\gcd(2p+1,2q+1)=1,$ then \begin{align*}
        3 \langle \mathcal C_\alpha, \mathcal C_\beta \rangle = \frac 1 {(2p+1)^2(2q+1)^2} \quad \text{ and } \quad
        3 \langle \mathcal S_\alpha, \mathcal S_\beta \rangle = \frac {(-1)^{p+q}} {(2p+1)^2(2q+1)^2}.  
        \end{align*}
        \item If $\alpha, \beta$ are not oddly co-linear, then \[\langle \mathcal C_\alpha, \mathcal C_\beta \rangle = 0 = \langle \mathcal S_\alpha, \mathcal S_\beta \rangle.\]
        \item Further, \[\langle \mathcal C_\alpha, \mathcal S_\beta \rangle = 0.\]
    \end{enumerate}
\end{theorem}
\begin{proof}
We will show this theorem using induction over $d \in \N.$

We start with the case $d=1$ for which we are going to use \cite[Lem. 2.1]{4}. We will first prove that $\alpha,\beta \in \N$ are oddly co-linear, if and only if, $\nu_2(\alpha)=\nu_2(\beta),$ where $\nu_2(n)$ denotes the exponent corresponding to 2 in the prime factorisation of $n.$

If $\alpha = \frac{2p+1}{2q+1} \beta,$ then $\nu_2(\alpha)=\nu_2(\alpha (2q+1)) = \nu_2(\beta (2p+1)) = \nu_2(\beta).$ On the other hand, if $\nu_2(\alpha)=\nu_2(\beta),$ then $\nu_2(\gcd(\alpha,\beta)) = \nu_2(\alpha)=\nu_2(\beta)$ and $2p+1:= \frac{\alpha}{\gcd(\alpha, \beta)}$ as well as $2q+1:= \frac{\beta}{\gcd(\alpha,\beta)}$ are odd, co-prime numbers such that $\alpha = \frac{ \ \frac{\alpha}{\gcd(\alpha,\beta)} \ }{\frac{\beta}{\gcd(\alpha,\beta)}} \beta = \frac{2p+1}{2q+1} \beta.$ Therefore, due to \cite[Lem. 2.1]{4} it holds $$3|\langle \mathcal S_\alpha, \mathcal S_\beta \rangle|=3 \langle \mathcal C_\alpha, \mathcal C_\beta \rangle = \frac{\gcd(\alpha,\beta)^4}{\alpha^2\beta^2} = \frac{1}{\left(\frac{\alpha}{\gcd(\alpha,\beta)}\right)^2\left(\frac{\beta}{\gcd(\alpha,\beta)}\right)^2} = \frac{1}{(2p+1)^2(2q+1)^2}$$ and $$\operatorname{sgn}(\langle \mathcal S_\alpha, \mathcal S_\beta \rangle) = - (-1)^{\frac{\alpha + \beta}{2\gcd(\alpha,\beta)}} = (-1)^{\frac 1 2(2p+1+2q+1)-1} = (-1)^{p+q}$$ if $\alpha$ and $\beta$ are oddly co-linear. 

Else we have $\nu_2(\alpha) \neq \nu_2(\beta)$ and the scalar products $\langle \mathcal C_\alpha, \mathcal C_\beta \rangle, \langle \mathcal S_\alpha, \mathcal S_\beta \rangle$ vanish according to \cite[Lem. 2.1]{4}(iii). Similarly all mixed scalar products $\langle \mathcal C_\alpha, \mathcal S_\beta \rangle$ vanish regardless of $\nu_2(\alpha), \nu_2(\beta)$ because of \cite[Lem. 2.1]{4}(i).

Now assume the assertion is proved for $d \in \N$ and let $\alpha,\beta \in \Z^{d+1}_+.$ Given a vector $z \in \R^{d+1}$ let $z' \in \R^d$ denote the truncated vector with $z'_j = z_j$ for all $1 \leq j \leq d.$

If $\alpha_{d+1}=0=\beta_{d+1},$ then $\alpha',\beta' \in \Zdpos$ as well as $\alpha^Tx=(\alpha')^Tx, \beta x' = (\beta')^Tx'$ hold true for all $x \in [0,1]^{d+1}$ and by Fubini's Theorem it holds $$\langle \mathcal C_\alpha, \mathcal C_\beta \rangle = \int_{[0,1]^d} \left(\int_{[0,1]} \mathcal C(\alpha^T x) \mathcal C(\beta^T x) \ \mathrm d x_{d+1} \right) \mathrm d x' = \langle \mathcal C_{\alpha'}, \mathcal C_{\beta'}\rangle$$ and analogously for $\mathcal S$ and mixed products. In this case, the theorem directly follows from assumption, as $\alpha$ and $\beta$ are oddly co-linear, if and only if, $\alpha'$ and $\beta'$ are. 

Otherwise we may assume $\alpha_{d+1} \neq 0$ because of symmetry. Let $g:=\gcd(\alpha_d,\alpha_{d+1}),$ then $g \in \N$ and by Euclid's Extended Algorithm we can find integers $u,v \in \Z$ such that $g=u \alpha_d + v \alpha_{d+1}.$ Let $M$ be the block matrix $$M := \begin{pmatrix}
    1 \!\!\! 1_{d-1} & 0_{d-1,2} \\ 0_{2,d-1} & A 
\end{pmatrix}$$ where $A:=\begin{pmatrix}
        u & v \\ -\frac{\alpha_{d+1}}{g} & \frac{\alpha_d}{g}
    \end{pmatrix}$, then all entries of $M$ are integers and $$\det M = \det A = \frac{u \alpha_d + v \alpha_{d+1}}{g}=1.$$ Further, we set $J:= M \alpha = (\alpha_1, ..., \alpha_{d-1}, g, 0)^T, K := M \beta \in \Z^{d+1}_+.$
    Since the integrand is $1$-periodic in every coordinate and $M^T$ is an integer matrix with $\det M^T=1$, the substitution $x\mapsto M^Tx$ maps $[0,1]^{d+1}$ onto a fundamental domain of $\Z^{d+1}$, so by the change of variables formula and periodicity the integral does not change. Using $\alpha^TM^Tx=J^Tx$ and $\beta^TM^Tx=K^Tx$ together with Fubini's Theorem we obtain \begin{align*} \int_{[0,1]^{d+1}} \mathcal C(\alpha^T x) \mathcal C(\beta^T x) \mathrm d x &= \int_{[0,1]^{d+1}} \mathcal C(\alpha^T M^T x) \mathcal C(\beta^T M^T x) \mathrm d x \\ &=\int_{[0,1]^{d+1}} \mathcal C(J^T x) \mathcal C(K^T x) \mathrm d x = \int_{[0,1]^d} \mathcal C_{J'} (x') \left(\int_0^1 \mathcal C((K')^T x' +K_{d+1} x_{d+1}) \mathrm d x_{d+1} \right) \mathrm d x'. \end{align*}
    Here, the 1-periodicity of $\mathcal C$ is used, to ensure $\mathcal C(\alpha^T(x+i)) = \mathcal C(\alpha^Tx + \alpha^Ti) = \mathcal C(\alpha^Tx)$ holds for all $x \in \R^d, i \in \Z^d$ as $\alpha^Ti = \sum \limits_{j=1}^d\alpha_j i_j \in \Z.$ Similarly for $\mathcal S$ and mixed products.
    
    Now there are two cases. First consider the case $K_{d+1} \neq 0.$ Then the inner integral vanishes and consequently the scalar product is 0. If $K_{d+1}=0$ however, then the integrand of the inner integral is constant and the value is $\mathcal C((K')^Tx')=\mathcal C_{K'}(x')$ leading to $\langle \mathcal C_\alpha, \mathcal C_\beta \rangle = \langle \mathcal C_{J'}, \mathcal C_{K'} \rangle.$ Summarised this reads as $$\langle \mathcal C_\alpha, \mathcal C_\beta \rangle = \begin{cases}
        0, &\text{ if } K_{d+1} \neq 0, \\ \langle \mathcal C_{J'}, \mathcal C_{K'} \rangle, &\text{ if } K_{d+1}=0,
    \end{cases}$$ where $K_{d+1}=\frac{-\alpha_{d+1} \beta_d + \alpha_d \beta_{d+1}}{g} \in \Z$ as $g=\gcd(\alpha_d, \alpha_{d+1})$ divides $\alpha_d$ and $\alpha_{d+1}$ and consequently the numerator. A completely analogous calculation shows the same result for $\mathcal S$ as well as mixed products. 
    
    We are now able to show the desired formulae. 
    \begin{itemize}
        \item We have seen $\langle \mathcal C_\alpha, \mathcal S_\beta \rangle = 0,$ if $K_{d+1} \neq 0.$ If $K_{d+1} = 0$ this follows from the induction hypothesis and $\langle \mathcal C_\alpha, \mathcal S_\beta \rangle = \langle \mathcal C_{J'}, \mathcal S_{K'} \rangle = 0$. Similarly $\langle \mathcal C_\beta, \mathcal S_\alpha \rangle=0.$
        \item Let $\alpha = \frac{2p+1}{2q+1} \beta,$ then $J= M \alpha = \frac{2p+1}{2q+1} M \beta = \frac{2p+1}{2q+1} K$ and $J' = \frac{2p+1}{2q+1}K'$ hold. If one of $\alpha_1, ..., \alpha_{d-1}$ is positive, we have $J \in \Z^{d+1}_+$ by definition. Otherwise $\alpha_1, ..., \alpha_{d-1}$ vanish and $g>0$ is the first non-zero entry. In either case it holds $J \in \Z^{d+1}_+$ and $K=\frac{2q+1}{2p+1} J \in \Z^{d+1}_+$. From $K_{d+1} = \frac{2q+1}{2p+1} J_{d+1} = 0,$ we get $$3 \langle \mathcal C_\alpha, \mathcal C_\beta \rangle = 3 \langle \mathcal C_{J'}, \mathcal C_{K'}  \rangle = \frac{1}{(2p+1)^2(2q+1)^2} \quad \text{and} \quad 3 \langle \mathcal S_\alpha, \mathcal S_\beta \rangle = 3 \langle \mathcal S_{J'}, \mathcal S_{K'} \rangle = \frac{(-1)^{p+q}}{(2p+1)^2(2q+1)^2}.$$
        \item Assume now $\alpha, \beta$ not to be oddly co-linear, then there are three cases to be distinguished: 
        \begin{itemize}
        \item[-] If $K_{d+1} \neq 0$ we are done, since the integrals vanish. 
        \item[-] If $K_{d+1}=0$ we again get $\langle \mathcal C_\alpha, \mathcal C_\beta \rangle = \langle \mathcal C_{J'}, \mathcal C_{K'} \rangle$ and exactly the same for $\mathcal S.$ If $(\alpha_1, ..., \alpha_{d-1}), (\beta_1, ..., \beta_{d-1})$ are not oddly co-linear, then $J'$ and $K'$ are not as well and we again get $\langle \mathcal C_\alpha, \mathcal C_\beta \rangle=0$ and $\langle \mathcal S_\alpha, \mathcal S_\beta \rangle=0$. 
        \item[-] It remains to show, that in the last case $(\alpha_1, ..., \alpha_{d-1}) = \frac{2p+1}{2q+1} (\beta_1, ..., \beta_{d-1})$ and $K_{d+1}=0$ the vectors $J',K'$ are not oddly co-linear. Because the first $d-1$ entries coincide with those from $\alpha$ and $\beta,$ we need to prove $g \neq \frac{2p+1}{2q+1} K_d.$
        
        Since $\alpha_{d+1} \neq 0$ was assumed, $\frac{\alpha_d \beta_{d+1}-\alpha_{d+1} \beta_d}{g}=K_{d+1}=0 \text{ implies } \ \alpha_d \frac{\beta_{d+1}}{\alpha_{d+1}}=\beta_d.$
        
        We now distinguish two cases again, namely $\beta_d=0$ and $\beta_d \neq 0.$ First, let us assume $\beta_d$ being zero, then either $\alpha_d=0$ or $\beta_{d+1}=0$ need to hold.
        
        If $\alpha_d=0,$ it follows $\alpha_d = 0 = \frac{2p+1}{2q+1} \cdot 0 = \frac{2p+1}{2q+1} \beta_d.$ We get $\alpha_{d+1} \neq \frac{2p+1}{2q+1} \beta_{d+1}$ because $\alpha, \beta$ are not oddly co-linear. Further, we have $$g=\alpha_{d+1}, \ u=0, \ v=1, \ K_d=u \beta_d + v \beta_{d+1} = \beta_{d+1},$$ yielding $g=\alpha_{d+1}\neq \frac{2p+1}{2q+1}\beta_{d+1}=\frac{2p+1}{2q+1}K_d.$ Second, if $\beta_{d+1}=0,$ then $K_d=0$ and $\frac{2p+1}{2q+1}K_d = 0 \neq g$ ending the proof of the first subcase. 
        
        We may now assume $\beta_d \neq 0$ and rearrange $\alpha_d \frac{\beta_{d+1}}{\alpha_{d+1}}=\beta_d$ to $\alpha_{d+1}=\frac{\beta_{d+1}}{\beta_d}\alpha_d.$
        
        Finally, assume aiming towards a contradiction, that $g=\frac{2p+1}{2q+1}K_d,$ then $$u \alpha_d + v \alpha_{d+1} = g = \frac{2p+1}{2q+1} K_d = u \frac{2p+1}{2q+1} \beta_d + v \frac{2p+1}{2q+1}\beta_{d+1} $$ holds, which in turn implies $$ v \left(\alpha_{d+1}-\frac{2p+1}{2q+1} \beta_{d+1}\right)=u \left(\frac{2p+1}{2q+1} \beta_d-\alpha_d \right).$$
        Plugging in $\alpha_{d+1}=\frac{\beta_{d+1}}{\beta_d}\alpha_d$ and factoring out $\frac{\beta_{d+1}}{\beta_d},$ leads to $$v \frac{\beta_{d+1}}{\beta_d}\left(\alpha_d-\frac{2p+1}{2q+1}\beta_d\right) = u \left(\frac{2p+1}{2q+1} \beta_d-\alpha_d\right).$$ Since $\alpha_d = \frac{2p+1}{2q+1}\beta_d$ would imply $\alpha_{d+1} =\frac{\alpha_d}{\beta_d} \beta_{d+1} = \frac{2p+1}{2q+1} \beta_{d+1}$ and consequently $\alpha=\frac{2p+1}{2q+1} \beta,$ we get $\alpha_d-\frac{2p+1}{2q+1} \beta_d \neq 0$ and may divide both sides by this term. We obtain $v \frac{\beta_{d+1}}{\beta_d} = -u \text{ implying } K_d = u \beta_d + v \beta_{d+1} = 0$ and $g = \frac{2p+1}{2q+1}K_d = 0,$ a contradiction to $g \in \N$, showing $g \neq \frac{2p+1}{2q+1}K_d$ as desired.
        
        Therefore $J',K'$ are not oddly co-linear and $K_{d+1}=0$ implies $$\langle \mathcal C_\alpha, \mathcal C_\beta \rangle = \langle \mathcal C_{J'}, \mathcal C_{K'} \rangle=0 \text{ as well as } \langle \mathcal S_\alpha, \mathcal S_\beta \rangle = \langle \mathcal S_{J'}, \mathcal S_{K'} \rangle=0.$$
        \end{itemize}
    \end{itemize}
\end{proof}

Recall the definition of symmetric groups $\{\sigma:\{1,...,n\} \to \{1,...,n\} \text{ bijective}\}$ which usually are denoted by $S_n$ but will be written as $P_n$ to avoid confusion with $\mathcal S$.

\begin{lemma}[\textbf{Similarity of sub-matrices}]\label{Streichmat}
    Let $k,n \in \N$ such that $n \geq2, 1 \leq k \leq n.$ Further, let $A,B,C,D\in \R^{n \times n}$ such that $D$ is a diagonal matrix with diagonal entries $D_{ii}$ in $\{-1,1\}.$ We define $C^{\dagger k} \in \R^{(n-1) \times (n-1)}$ to be the matrix gained by deleting the $k$-th row and column from $C,$ that is $$C^{\dagger k}_{ij} = \begin{cases}
    C_{ij} &\text{ if } 1 \leq i,j < k,\\
    C_{i(j+1)} &\text{ if } 1 \leq i < k \leq j \leq n-1,\\
    C_{(i+1)j} &\text{ if } 1 \leq j < k \leq i \leq n-1,\\
    C_{(i+1)(j+1)} &\text{ if } k \leq i,j \leq n-1. \\
\end{cases}$$ 
If $B=DAD,$ then $B^{\dagger k} = D^{\dagger k} A^{\dagger k} D^{\dagger k}.$
\end{lemma}
\begin{proof} Note that $$B_{ij}=(DAD)_{ij}=\sum \limits_{k=1}^n (DA)_{ik} D_{kj} = (DA)_{ij} D_{jj} = \sum \limits_{k=1}^n D_{ik} A_{kj} D_{jj} = D_{ii}D_{jj}A_{ij}$$ holds true for all $1 \leq i,j \leq n.$ Similarly, $D^{\dagger k}$ is also diagonal, such that for all $1 \leq i,j \leq n-1$ it holds \begin{align*} (D^{\dagger k} A^{\dagger k} D^{\dagger k})_{ij}=D^{\dagger k}_{ii} D^{\dagger k}_{jj} A^{\dagger k}_{ij} = \begin{cases}
    D_{ii}D_{jj}A_{ij} &\text{ if } 1 \leq i,j < k,\\
    D_{ii}D_{(j+1)(j+1)}A_{i(j+1)} &\text{ if } 1 \leq i < k \leq j,\\
    D_{(i+1)(i+1)}D_{jj}A_{(i+1)j} &\text{ if } 1 \leq j < k \leq i,\\
    D_{(i+1)(i+1)}D_{(j+1)(j+1)}A_{(i+1)(j+1)} &\text{ if } k \leq i,j, \\
\end{cases} = (DAD)^{\dagger k}_{ij} = B^{\dagger k}_{ij}. \end{align*}
\end{proof}

Note that it holds $D^{-1}=D$ and $(D^{\dagger k})^{-1}=D^{\dagger k}$ as $D^{\dagger k}$ is once again diagonal with entries in $\{-1,1\}.$ Hence $B$ is similar to $A$ and $B^{\dagger k}$ is similar to $A^{\dagger k}$.

The explicit form of the scalar products shows that the Gram matrices decompose into blocks which are indexed by classes of oddly co-linear vectors. This allows us to reduce the comparison of the cosine and the sine system to the one-dimensional case.

\begin{theorem}[\textbf{Spectra of the truncated Gram matrices}] \label{Gram}
Fix $d \in \N.$ Let $M = \{\alpha_1, ..., \alpha_{|M|}\} \subset \{\alpha \in \Zdpos\}$ be non-empty and finite. Let $G^\mathcal C_M$ be the matrix with $(G^\mathcal C_M)_{ij}:=\langle \mathcal C_{\alpha_i}, \mathcal C_{\alpha_j}\rangle$ and define $G^\mathcal S_M$ analogously. Then $G^\mathcal C_M$ is similar to $G^\mathcal S_M$ and, in particular, their spectra coincide. 
\end{theorem}
\begin{proof}
First we notice that odd co-linearity defines an equivalence relation on $M.$
    \begin{itemize}
        \item Reflexivity: For any $\alpha \in M$ we choose $p=0=q$ to get $\alpha = \frac{1}{1} \alpha.$
        \item Symmetry: Let $\alpha, \beta \in M$ and $\alpha=\frac{2p+1}{2q+1} \beta,$ then $\beta = \frac{2q+1}{2p+1} \alpha.$
        \item Transitivity: Let $\alpha, \beta, \gamma \in M$ with $\alpha = \frac{2p+1}{2q+1} \beta, \beta = \frac{2r+1}{2s+1} \gamma,$ then $\alpha = \frac{(2p+1)(2r+1)}{(2q+1)(2s+1)}\gamma.$ Numerator and denominator of this fraction are odd as products of odd numbers and simplifying gives co-prime odd numbers as numerator and denominator.
    \end{itemize}

Therefore, $M$ is partitioned into finitely many pairwise disjoint equivalence classes. We may order $M$ in any way and will do so by grouping equivalence classes. To be more precise: 

Let $Q_1=\{\alpha^1_1,  ..., \alpha^1_{|Q_1|}\}, Q_2=\{\alpha^2_1, ..., \alpha^2_{|Q_2|}\}, ..., Q_K=\{\alpha^K_1, ..., \alpha^K_{|Q_K|}\}$ be the equivalence classes arising from odd co-linearity. 
Because of Theorem \ref{SKP2}, the resulting matrices are in block-diagonal form and we can write $G^\mathcal C_{M,k} = (\langle \mathcal C_{\alpha_i^k}, \mathcal C_{\alpha_j^k} \rangle)_{i,j \in \{1,..., |Q_k|\}}$ and define $G^\mathcal S_{M,k}$ analogously for all $1 \leq k \leq K$. Then $G^\mathcal C_M$ is similar to $$\begin{pmatrix}
    G^\mathcal C_{M,1} & 0 & ... & 0 \\
    0 & G^\mathcal C_{M,2} & ... & 0 \\ 
    \vdots & \vdots & \ddots & 0 \\ 0 & 0 & ... & G^\mathcal C_{M,K}
\end{pmatrix}$$ and all the same for $\mathcal S.$ 

For any equivalence class $Q_k$ and any $l \in \{1, ..., d\},$ $\alpha_l=0$ holds either for all $\alpha \in Q_k$ or for none. Since $\alpha \in \Zdpos$ there has to be a smallest $l \in \{1, ..., d\}$ such that $\alpha_l > 0$ for all $\alpha \in Q_k.$

Since the value of $\langle \mathcal C_{\alpha_i^k}, \mathcal C_{\alpha_j^k} \rangle$ only depends on the integers defining the co-linearity scalar which themselves are uniquely determined by the $l$-th entries of the vectors as they are assumed to be co-prime, due to Theorem \ref{SKP2}(i) and (ii), we get $\langle \mathcal C_{\alpha_i^k}, \mathcal C_{\alpha_j^k} \rangle=\langle \mathcal C_{(\alpha_i^k)_l}, \mathcal C_{(\alpha_j^k)_l} \rangle$ and equally for $\mathcal S.$ Thus $G^\mathcal C_{M,k}$ is equal to a block sub-matrix of the one-dimensional Gram matrix and the same is true for $G_{M,k}^\mathcal S.$ These are similar because of \cite[Thm. 2.2]{4} and Lemma \ref{Streichmat}. Therefore, $G_M^\mathcal C$ and $G_M^\mathcal S$ are similar as well, which in particular implies their spectra to coincide. Note that the transformation matrix is again block diagonal, where each block is a diagonal matrix with entries in $\{-1,1\}.$ Therefore, $DG_M^\mathcal C D = G_M^\mathcal S$ holds for some diagonal matrix $D$ with $D_{jj} \in \{-1,1\}$ for all $1 \leq j \leq |M|.$
\end{proof}

We now come to the main result of this section. The Fourier coefficients of $\sqrt3\,\mathcal C$ and $\sqrt3\,\mathcal S$ define two sequences $\tau$ and $\lambda$ for which $\mathcal R_d$ is a mixed system. Since $\frac1{\tau_1}\tau$ and $\frac1{\lambda_1}\lambda$ are completely multiplicative, Theorem \ref{Basiscrit}(ii) applies.  
As a result, we obtain the lower Riesz constant $A=\frac 2 3$ for $\mathcal R_d$. This improves the constant $\frac 1 2$ from \cite{mathcalCS,2} and the constant $\frac{96}{\pi^4}\big(2-\frac{\pi^2}{8}\big)^2\approx 0.579$ from \cite{4}, while the upper constant $B=\frac 3 2$ is the same as before. All constants are independent of the dimension $d$.

\begin{theorem}[\textbf{Basis properties of $\mathcal R_d$}]\label{Rdbasis}
    Let $q \in (1,\infty)$ and $d \in \N.$ The system $\mathcal R_d$ forms a Schauder basis of $L_q([0,1]^d).$ In case of $q=2,$ it forms a Riesz basis where $A=\frac 2 3$ and $B=\frac 3 2.$
\end{theorem}

\begin{proof}
    We will first prove the formula $\sum \limits_{n=1}^\infty \frac{|\mu(2n-1)|}{(2n-1)^2}=\frac{12}{\pi^2},$ where $\mu$ is the Möbius function. Let $L: \N \to \{-1,1\}$ be the Liouville function, that is introduced in \cite[Ch. 2.12]{Numertheory}. There, it is shown, that $L$ is completely multiplicative and that it is the Dirichlet inverse of $|\mu|$. 
    
    Since $L$ and $|\mu|$ are bounded by 1 the Dirichlet series $\zeta(s):=\sum \limits_{n=1}^\infty \frac 1 {n^s}$, $\xi(s):=\sum \limits_{n=1}^\infty \frac{L(n)}{n^s}$ and $\eta(s):=\sum \limits_{n=1}^\infty \frac{|\mu(n)|}{n^s}$ converge for all $s \in \C$ with $\text{Re}(s)>1.$ In \cite[Ex. 6 in Ch. 11.4]{Numertheory} the formula $\eta(s)=
    \frac{\zeta(s)}{\zeta(2s)}$ for all $s \in \C$ such that $\text{Re}(s)>1$ is derived. Moreover, for all $s \in \C$ such that $\text{Re}(s)>1$ it holds $\xi(s) \eta(s)=\sum \limits_{n=1}^\infty \frac{(L*|\mu|)(n)}{n^s}=\sum \limits_{n=1}^\infty \frac{e(n)}{n^s}=1$ and $\xi(s)=\frac 1 {\eta(s)}=
    \frac{\zeta(2s)}{\zeta(s)}.$ 
 Now $\sum \limits_{n=1}^\infty \frac{|\mu(n)|}{n^2}=
 \eta(2)=\frac{\zeta(2)}{\zeta(4)}=\frac{15}{\pi^2}$ follows. Further, \begin{align*}\frac{15}{\pi^2}=\sum \limits_{n=1}^\infty \frac{|\mu(n)|}{n^2}&=\sum \limits_{n=1}^\infty \frac{|\mu(2n)|}{(2n)^2}+\sum \limits_{n=1}^\infty \frac{|\mu(2n-1)|}{(2n-1)^2} \\ &=\sum \limits_{n \in 2\N} \frac{|\mu(2n)|}{(2n)^2}+\sum \limits_{n \in 2\N-1} \frac{|\mu(2n)|}{(2n)^2}+\sum \limits_{n \in 2\N-1} \frac{|\mu(n)|}{n^2} \\ &= \frac{|\mu(2)|}{2^2}\sum \limits_{n \in 2\N-1} \frac{|\mu(n)|}{n^2}+\sum \limits_{n \in 2\N-1} \frac{|\mu(n)|}{n^2}=\frac 5 4 \sum \limits_{n \in 2\N-1} \frac{|
 \mu(n)|}{n^2} \end{align*} holds true, 
 because $2^2 \mid 2n$ for even $n$ implies $\mu(2n)=0$ in this case, and for odd $n$ the multiplicativity of $\mu$ together with $\gcd(2,n)=1$ yields $|\mu(2n)|=|\mu(2)|\,|\mu(n)|=|\mu(n)|$.
 Thus, the desired formula is achieved after division by $\frac 5 4.$

    By \cite[p.166]{mathcalCS} we see that it holds $$\sqrt 3 \mathcal C = \sqrt 2 \sum \limits_{k=1}^\infty \frac{4 \sqrt 6}{\pi^2} \frac{1}{(2k-1)^2}c_{2k-1}=C^\tau \quad \text{ and } \quad \sqrt 3 \mathcal S=\sqrt 2 \sum \limits_{k=1}^\infty \frac{4 \sqrt 6}{\pi^2} \frac{(-1)^{k+1}}{(2k-1)^2} s_{2k-1}=S^\lambda,$$ where $$\tau_k:= \begin{cases}
        0, &\text{ if } k \text{ is even, } \\ \frac{4\sqrt 6}{\pi^2} \frac 1 {k^2}, &\text{ if } k \text{ is odd,}
    \end{cases}$$ and 
    $\lambda_{2k-1}=(-1)^{k+1}\tau_{2k-1}$. Both sequences $\tau$ and $\lambda$ satisfy the hypotheses of Theorem \ref{Basiscrit}(ii) and Theorem \ref{MixedZystems}. 
In particular, using Theorem \ref{Basiscrit}(ii)  with $\tau_1=\lambda_1=\frac{4\sqrt6}{\pi^2}$, $\|\tau\|_{\ell_1}=\|\lambda\|_{\ell_1}=\frac{4\sqrt6}{\pi^2}\cdot\frac{\pi^2}{8}=\frac{\sqrt6}{2}$ and $\|\mu\cdot\tau\|_{\ell_1}=\|\mu\cdot\lambda\|_{\ell_1}=\frac{48\sqrt6}{\pi^4}$ we obtain
$$A_\tau=A_\lambda=\frac 1 {\|\tau^{-1}\|_{\ell_1}^2}=\frac{|\tau_1|^4}{\|\mu\cdot\tau\|_{\ell_1}^2}=\frac{(4\sqrt6)^4\pi^{-8}}{(48\sqrt6)^2\pi^{-8}}=\frac23,
\quad \text{ and } \quad B_\tau=B_\lambda=\|\tau\|_{\ell_1}^2=\frac32 .$$
Theorem \ref{MixedZystems} with $k_0=1$ now yields $A=\min\{1,\frac23\}=\frac23$ and $B=\max\{1,\frac32\}=\frac32$.
\end{proof}

\textbf{Acknowledgement:}
The first author thanks Christopher Sorg for some valuable comments when preparing this manuscript. 

{\bf Declaration of generative AI use:} Claude Opus 5.5.

\bibliographystyle{alpha}

\bibliography{bibliography}


\end{document}